\documentclass[11pt,a4paper,reqno]{amsart}
\usepackage{vmargin,color}

\usepackage[english]{babel}
\usepackage[utf8]{inputenc}

\usepackage{bbm,amsmath,amsthm,epsfig,latexsym,marvosym}
\usepackage{amsfonts}
\usepackage{bigints}
\usepackage{amsmath}
\usepackage{amssymb}

\def\R{\mathbb R}
\def\N{\mathbb N}

\def\cal{\mathcal}

\def\E{{\cal E}}
\def\F{{\cal F}}

\def\H{{\cal H}}

\def\L{{\cal L}}

\def\a{\alpha}

\def\de{\delta}
\def\e{\varepsilon}
\def\l{\lambda}
\def\om{\omega}

\def\vt{\vartheta}

\def\Om{\Omega}
\def\La{\Lambda}
\def\S{\Sigma}

\newenvironment{sistema}%
{\left\lbrace\begin{array}{@{}l@{}}}%
{\end{array}\right.}

\def\pa{\partial}

\def\d{\, \mathrm{d}}

\def\bd{{\rm bd}}

\def\ca{\mathbbmss{1}}

\def\pared{\partial^{*}}

\def\00{{\bf 0}}
\def\dive{{\rm div}}
\def\mez{\left(\frac{1}{2}\right)}
\def\uno{(1)}
\def\zero{(0)}

\newcommand{\cc}{\subset\subset}

\DeclareMathOperator*{\diam}{diam}

\newtheorem{theorem}{Theorem}[section]
\newtheorem{corollary}[theorem]{Corollary}
\newtheorem{proposition}[theorem]{Proposition}
\newtheorem{lemma}[theorem]{Lemma}

\theoremstyle{definition}
\newtheorem{remark}[theorem]{Remark}
\newtheorem{definition}[theorem]{Definition}

\numberwithin{equation}{section}
\numberwithin{figure}{section}

\author{M. Caroccia}
\address{Dipartmento di Matematica, Universit\`a di Pisa, Largo Bruno Pontecorvo 5, 56127 Pisa, Italy}
\email{caroccia.marco@gmail.com}

\title{Cheeger $N$-clusters}

\begin{document}
\begin{abstract}
In this paper we introduce a Cheeger-type constant defined as a minimization of a suitable functional among all the $N$-clusters contained in an open bounded set $\Omega$. Here with $N$-Cluster we mean a family of $N$ sets of finite perimeter, disjoint up to a set of null Lebesgue measure. We call any $N$-cluster attaining such a minimum a \textit{Cheeger $N$-cluster}. Our purpose is to provide a non trivial lower bound on the optimal partition problem for the first Dirichlet eigenvalue of the Laplacian. Here we discuss the regularity of Cheeger $N$-clusters in a general ambient space dimension and we give a precise description of their structure in the planar case. The last part is devoted to the relation between the functional introduced here (namely the \textit{$N$-Cheeger constant}), the partition problem for the first Dirichlet eigenvalue of the Laplacian and the Caffarelli and Lin's conjecture.
\end{abstract}
\maketitle
\tableofcontents

\section{Introduction}\label{cpt 4 sct 1}
For a given open, bounded set $\Om$ and an integer $N\in \N$ we introduce the \textit{$N$-Cheeger constant of $\Om$} as:
\begin{equation}\label{N-cheeger constant 1} 
H_N(\Om)=\inf\left\{\sum_{i=1}^N\frac{P(\E(i))}{|\E(i)|}\ \Big{|}\ \E=\{\E(i)\}_{i=1}^{N} \subseteq \Om , \text{ is an $N$-cluster}\right\}.
\end{equation}
where $|\cdot|$ denotes the Lebesgue measure and $P(\cdot)$ denotes the distributional perimeter.\\

We do not want to enter into the details of what a \textit{finite perimeter set} is (for a complete overview on such a topic we refer the reader to \cite[Chapters 12-20]{maggibook}) let us just highlights that the distributional perimeter of a Borel set $E$ having smooth boundary $\pa E$ coincide with the $(n-1)$-dimensional area of the boundary: $\H^{n-1}(\pa E)$. \\

Here and in the sequel an $N$-cluster $\E$ should be intended just as a family of $N$ Borel sets $\E=\{\E(i)\}_{i=1}^N$ with the following properties:
\begin{itemize}
\item[(a)] $0<|\E(i)|<+\infty$ for all $i=1,\ldots,N$;
\item[(b)] $|\E(i)\cap \E(j)|=0$ for all $i,j=1,\ldots,N$ with $i\neq j$;
\item[(c)] $P(\E(i))<+\infty$ for all $i=1,\ldots,N$.
\end{itemize}

As shown below in Theorem \ref{existence}, the infimum in \eqref{N-cheeger constant 1} is always attained and we refer to the minimizers as the \textit{Cheeger $N$-clusters of $\Om$}.\\

\indent We focus on the quantity $H_N$ because it seems to represent the right object to study in order to provide some non trivial lower bound on the optimal partition functional
\begin{equation}\label{p-laplacian}
\Lambda_N^{(p)}(\Om)=\inf\left\{\sum_{i=1}^N\lambda_1^{(p)}(\E(i))  \right\},
\end{equation}
where $\lambda_1^{(p)}$ denotes the first Dirichlet eigenvalue of the p-Laplacian, defined as:
$$\lambda_1^{(p)}(E):=\inf\left\{\int_{E} |\nabla u|^{p} \d x \ \Big{|}\ u\in W^{1,p}_0(E), \ \|u\|_{L^{p}}=1 \right\}.$$
The infimum in \eqref{p-laplacian} is taken over all the $N$-clusters $\E$ whose chambers are \textit{quasi-open sets of $\Om$}. The family of quasi-open sets of an open bounded set $\Om$ is a suitable sub-class of the Borel's algebra of $\Om$  where the first Dirichlet eigenvalue of the $p$-Laplacian $\l_1^{(p)}$ can be defined. The definition of quasi-open set is related to the concept of $p$-capacitary measure in $\R^n$ that we do not need to recall in here (see \cite{EvGa91} for more details about it). For our purposes it is enough to recall that:
\text{}\newline
 \begin{center}
\textit{the quasi-open sets are the upper levels of $W^{1,p}$ functions as well as the open sets are the upper levels of continuous functions. Each open set of an open bounded set $\Om$ is also a quasi-open set of $\Om$.}
\end{center}
\text{}\newline
The importance of the partition problem \eqref{p-laplacian} relies in the fact that it provides a way to look at the asymptotic behavior in $N$ of the $N$-th Dirichlet eigenvalue of the classical Laplacian (the $2$-Laplacian), as Caffarelli and Lin show in \cite{CaLi07}. The  $N$-th Dirichlet eigenvalue of the Laplacian of an open set $\Om$ is recursively defined as 
\begin{align*}
\l_N^{(2)}(\Om)&=\inf_{u\in X_{N-1}}\left\{\frac{\int_{\Om}|\nabla u|^2\d x}{\int_{\Om}|u|^2\d x}\right\}\\
X_{N-1}&=\left\{u\in W^{1,2}_0(\Om) \ | \ \langle u,u_i\rangle_2=0, \ \ \text{for all $i=1,\ldots, N-1$}\right\}
\end{align*}
where $u_1,\ldots,u_{N-1}$ are the first $N-1$ eigenfunctions
	\[
	\l_i^{(2)}(\Om)=\frac{\int_{\Om}|\nabla u_i|^2\d x}{\int_{\Om}|u_i|^2\d x} \ \ \ \ \ \text{for all $i=1,\ldots, N-1$}
	\]
and $\langle \cdot ,\cdot\rangle_2$ denotes the standard scalar product of $L^2(\Om)$
	\[
	\langle u ,v\rangle_2=\int_{\Om} uv\d x \ \ \ \ \ \text{for all $u,v\in L^2(\Om)$}
	\]
(see \cite[Section 6.5]{EvGa91} for a detailed discussion about eigenvalues and eigenfunctions). In \cite{CaLi07}, Caffarelli and Lin prove that there exist two constants $C_1$ and $C_2$ depending only on the dimension such that
\begin{equation}\label{CaLi estimate}
C_1\frac{\Lambda_N^{(2)}(\Om)}{N}\leq \lambda_N^{(2)}(\Om) \leq C_2 \frac{\Lambda_N^{(2)}(\Om)}{N},
\end{equation}
where $\lambda_N^{(2)}$ is the $N$-th Dirichlet eigenvalue. The detailed study of $\lambda_N^{(2)}(\Om)$ for $N\geq 2$ seems to be an hard task (so far only the case $N=1,2$ are well known in details, see for instance \cite{H06}) and that is why the asymptotic approach suggested by Caffarelli and Lin could be a good way to look at the spectral problem. We also refer the reader to \cite{bucur2012minimization} where the existence of minimizers for $\l_N^{(2)}$ is proved. \\

\indent Caffarelli and Lin's conjecture (appearing in \cite{CaLi07}) about the asymptotic behavior of $\Lambda_N^{(2)}(\Om)$ in the planar case states that 
$$\Lambda_N^{(2)}(\Om)=\frac{N^{2}}{|\Om|}\lambda_1^{(2)}(H)+o(N^{2}),$$
where $H$ denotes a unit-area regular hexagon. So far, no progress has been made in proving the conjecture, anyway numerical simulations (see \cite{BouBucO09}) point out that the conjecture could be true. If the conjecture turns out to be true, relation \eqref{CaLi estimate} could be improved, in the planar case, as:
\begin{equation}\label{CaLi estimate 2.0}
C_1\frac{N\lambda_1^{(2)}(H)}{|\Om|}+o(N)\leq \lambda_N^{(2)}(\Om) \leq C_2 \frac{N\lambda_1^{(2)}(H)}{|\Om|}+o(N).
\end{equation}

In order to explain the connection between $H_N$ and $\Lambda_N^{(p)}$ we recall some well-known fact about the classical Cheeger constant of a Borel set $\Om$:
\begin{equation}\label{cheeger constant}
h(\Om):=\inf\left\{\frac{P(E)}{|E|} \ \Big{|} \ E \subseteq \Om \right\},
\end{equation}
(note that $h(\Om)=H_1(\Om)$). Given an open set $\Om$, each set $E\subseteq \Om$ such that $h(\Om)=\frac{P(E)}{|E|}$ is called \textit{Cheeger set for $\Om$}. It is possible to prove that each Cheeger set $E$ for $\Om$ is a $(\La,r_0)$-perimeter-minimizing inside $\Om$ (see Definition \ref{Lambdarminimi} below) and that $\pared E\cap \Om$ is a constant mean curvature analytic hypersurface relatively open inside $\pa E$ (here $\pared E$ denotes the \textit{reduced boundary of the finite perimeter set $E$}, we refer the reader to \cite[Chapter 15]{maggibook} for further details). Furthermore, the mean curvature $C$ of the set $E$ in the open set $\Om$ is equal to $C=\frac{1}{n-1}h(E)$. We refer the reader to \cite{Pa11} and \cite{Leo15}: two exhaustive surveys on Cheeger sets and Cheeger constant.\\

The Cheeger constant was introduced by Jeff Cheeger in \cite{Ch70} and provides a lower bound on the first Dirichlet eigenvalue of the $p$-Laplacian of a domain $\Om$. By exploiting the coarea formula and Holdër' s inequality  it is possible to show that for every domain $\Om$ and for every $p>1$ it holds,
\begin{equation}\label{lb p-eig}
\lambda_1^{(p)}(\Om)\geq \left(\frac{h(\Om)}{p}\right)^p.
\end{equation}
The Cheeger constant is also called the {\it first Dirichlet eigenvalue of the 1-laplacian} since, thanks to \eqref{lb p-eig} and to a comparison argument
\begin{equation}\label{eqn limite per p che tende a uno}
\lim_{p\rightarrow 1} \lambda_1^{(p)}(\Om)=h(\Om).
\end{equation}
See, for example, \cite{KN08} for more details about the relation between the Cheeger constant and the first Dirichlet eigenvalue of the p-Laplacian  or \cite{BucBu05} and \cite{Bu10} for more details about the spectral problems and shape optimization problems. \\

We note here that the constant $H_N$ is the analogous of the Cheeger constant in the optimal partition problem for $p$-laplacian eigenvalues. We refer the reader to \cite{Pa09}, where a generalized type of Cheeger constant for the $2$-nd Dirichlet eigenvalue of the Laplacian is also studied. As we show in Proposition \ref{limit} below, we can always give a lower bound on $\Lambda_N^{(p)}$ by making use of \eqref{lb p-eig} and Jensen's inequality: 
\begin{equation}\label{Lp lowerbound}
\Lambda_N^{(p)}(\Om)\geq \frac{1}{N^{p-1}}\left(\frac{H_N(\Om)}{p}\right)^{p}.
\end{equation}
By combining \eqref{Lp lowerbound} with a comparison argument (see Theorem \ref{limit} below) we are also able to compute the limit as $p$ goes to $1$ and obtain
 	\begin{equation}\label{viva l italia}
 	\lim_{p\rightarrow 1}  \Lambda_N^{(p)}(\Om)=H_N(\Om).
 	\end{equation}
Thus, the constant $H_N$ seems to provide the suitable generalization of the Cheeger constant for the study of $\La_N^{(p)}$.\\

In this paper we mainly focus on the general structure and regularity of Cheeger $N$-clusters in order to lay the basis for future investigations on $H_N$. In the final section, once we have proved \eqref{viva l italia}, we study the asymptotic behavior of $H_N$ in the planar case. The statements involving regularity are quite technical and we reserve to them the whole Section \ref{cpt 4 sct 1 sbsct 1} (Theorems \ref{mainthm1}, \ref{mainthm2} and \ref{mainthm3}),  we just point out here that if $\E$ is a Cheeger $N$-cluster of $\Om$ the following statement holds.\\

\begin{large}
Theorems \ref{mainthm1}, \ref{mainthm2} and \ref{mainthm3}.\end{large} \textit{For every $i=1,\ldots,N$ the reduced boundary of each chambers $\pared \E(i)\cap \Om$ is a $C^{1,\a}-$h\-y\-per\-sur\-fa\-ces (for every $\a\in (0,1)$ ) that is relatively open inside $\pa \E(i)\cap \Om$. Furthermore it is possible to characterize the singular set of a Cheeger $N$-cluster $\E$ as a suitable collection of points with density zero for the external chamber 
	\[
	\E(0)=\Om\setminus \bigcup_{i=1}^N \E(i).
	\]
Moreover if the dimension is $n=2$ then the singular set is discrete and the chambers $\E(i)\cc \Om$ are indecomposable.}\\

Note that, in this context, the external chambers should be intended as $\Om\setminus (\cup_i \E(i) )$ instead of $\R^n\setminus \bigcup_i\E(i)$ as usual (that is because the ambient space is $\Om$ in place of $\R^n$). As we are pointing out below, also the definition of "singular set of a Cheeger $N$-cluster" must be given in a slightly different way (see \eqref{insieme singolare 1}) from the standard one $\pa \E \setminus \pared \E$, since this last set turns out to be too small. Let us postpone this discussion below to Section \ref{cpt 4 sct 1 sbsct 1} where precise statements are given, and let us, instead, briefly focus on the asymptotic properties of $H_N$ (to which Subsection \ref{asymptoyc of HN} is devoted).\\

We note that for $H_N$ it is reasonable to expect a behavior of the type 
\begin{equation}\label{come va}
H_N(\Om)= C(\Om) N^{\frac{3}{2}}+o(N^{\frac{3}{2}}), 
\end{equation}
for some constant $C(\Om)$. In Theorem \ref{asymptotic behavior} (Property 3) ) we provide some asymptotic estimate for $H_N$ showing that the exponent $\frac{3}{2}$ in \eqref{come va} is the correct one and proving that for any given bounded open set $\Om\subset \R^2$ it holds
\begin{equation}\label{asintotico in N per HN}
\frac{h(B)\sqrt{\pi}}{\sqrt{|\Om|}}\leq\liminf_{N\rightarrow +\infty} \frac{H_N(\Om)}{N^{\frac{3}{2}}}\leq  \limsup_{N\rightarrow +\infty} \frac{H_N(\Om)}{N^{\frac{3}{2}}}\leq  \frac{h(H)}{\sqrt{|\Om|}},
\end{equation}
We here conjecture that 
$$C(\Om)=\frac{h(H)}{\sqrt{|\Om|}},$$ 
which is nothing more than Caffarelli and Lin's conjecture for the case $p=1$. Note that, thanks to \eqref{Lp lowerbound} this would imply
\begin{equation}
\Lambda_N^{(2)}(\Om)\geq \frac{N^{2}}{|\Om|} \left(\frac{h(H)}{2}\right)^{2}+o(N^2),
\end{equation}
a "weak" version of Caffarelli and Lin's conjecture.  It seems coherent and natural to expect this kind of behavior for $H_N(\Om)$.\\

The paper is organized as follows. In Section \ref{cpt 4 sct 1 sbsct 1} we present and comment the three main statements describing the regularity property and the structure of Cheeger $N$-clusters. Sections \ref{cpt 4 sct Existence and regularity}, \ref{cpt 4 sct The singular set of the Cheeger N-clusters in low dimension} and \ref{cpt 4 the planar case} are devoted to the proof of the Theorems introduced in Section \ref{cpt 4 sct 1 sbsct 1}. In the final section \ref{limite} we show the connection between $H_N$ and $\La^{(p)}_N$ and we establish the asymptotic trend of $H_N$ (for $N$ sufficiently large) in the planar case.\\

{\bf \noindent Acknowledgments.} The author is grateful to professor Giovanni Alberti for his useful comments and for the useful discussions about this subject. Thanks to Gian Paolo Leonardi and Aldo Pratelli for have carefully reading this work as a part of the Ph. D thesis of the author. The author is also grateful to Enea Parini for the useful comments about Proposition \ref{ognicameraconfinaconilvuoto}. The work of the author was partially supported by the project 2010A2TFX2-Calcolo delle Variazioni, funded by\textit{ the Italian Ministry of Research and University}. 

\section{Basic definitions and regularity theorems for Cheeger $N$-clusters }\label{cpt 4 sct 1 sbsct 1}

We present three statements that we are going to prove in Section \ref{cpt 4 sct Existence and regularity} and in Subsections \ref{cpt 4 sbsct Proof of mainthm2} and \ref{cpt 4 sbsct proof of mainthm3}. \\

In Section \ref{cpt 4 sct Existence and regularity} after we have shown existence of Cheeger $N$-clusters for any given \textit{bounded} ambient space $\Om$ with finite perimeter (Theorem \ref{existence}) we provide the partial regularity Theorem \ref{mainthm1}. Set, for a generic Borel set $F$ and for $i=1,\ldots,N$
\begin{align}
\S(\E(i);F)&:=[\pa \E(i)\setminus \pared\E(i)]\cap F\label{insieme singolare i},\\
\S(\E(i))&:=\S(\E(i);\R^n).\label{insieme singolare ii}
\end{align}
\begin{theorem}\label{mainthm1}
Let $n\geq 1, N\geq 2$. Let $\Om\subset \R^n$ be an open bounded set with finite perimeter and $\E$ be a Cheeger $N$-cluster of $\Om$.
Then for every $i=1,\ldots,N$ the following statements hold true: 
\begin{itemize}
\item[(i)] For every $\a\in \left(0,1\right)$ the set $\Om \cap \pared \E(i)$ is a $C^{1,\a}$-hypersurface that is relatively open in $\Om\cap \pa\E(i)$ and it is $\H^{n-1}$ equivalent to $\Om \cap \pared \E(i)$; 
\item[(ii)] For every $i=1,\ldots,N$ the set $\pa \E(i) \cap \Om$ can meet $\pared \Om$ only in a tangential way, that is: $\pa^*\Om \cap \pa \E(i) \subseteq \partial^*\E(i)$. Moreover for every $x\in \pa^*\Om \cap \pa\E(i)$ it holds:
$$\nu_{\E(i)}(x)=\nu_{\Om}(x).$$
Here $\nu_{\E(i)}$, $\nu_{\Om}$ denote, respectively, the measure theoretic outer unit normal to $\E(j)$ and to $\Om$;
\item[(iii)]  $\S(\E(i);\Om)$ is empty if $n\leq 7$;
\item[(iv)] $\S(\E(i);\Om)$ is discrete if $n=8$;
\item[(v)] if $n\geq 9$, then $\H^{s}(\S(\E(i);\Om))=0$ for every $s>n-8$.
\end{itemize}  
\end{theorem}

For proving $(i),(iii),(iv)$ and $(v)$ we simply show (in Theorem \ref{regolare}) that each chamber $\E(i)$ is a $(\Lambda,r_0)$-perimeter-minimizing in $\Om$ (see Definition \ref{Lambdarminimi} below) and then we make use of the De Giorgi's regularity Theorem \ref{regularity}, retrieved below for the sake of completeness. We re-adapt an idea from \cite{BaMa82} based on the fact that a solution of an obstacle problem having bounded distributional mean curvature is regular. Assertion $(ii)$ follows as a consequence of \cite[Proposition 2.5, Assertion  (vii)]{LP14} retrieved below (Proposition \ref{leo}).

\begin{remark}
\rm{ We need to ask that $\Om$ is bounded otherwise no Cheeger $N$-clusters are attained. Indeed if $\Om$ is unbounded, by intersecting $\Om$ with $N$ suitable disjoint balls of radius approaching $+\infty$ we easily obtain $H_N(\Om)=0$.}
\end{remark}

\subsection{The role of the singular set $\S(\E)$}

Note that Theorem \ref{mainthm1} yields the inner regularity of \textit{all} the chambers, differently from the Theorems appearing in literature about regularity of isoperimetric $N$-clusters (see for example \cite[Corollary 4.6]{CiLeMaIC1}, \cite[Chapter 30]{maggibook}) that usually involves the topological boundary and the reduced boundary of the whole cluster 
\begin{align*}
\pa \E:=\bigcup_{i=1}^{N}\pa \E(i), \ \ \ \ \ \pared \E:=\bigcup_{0\leq h<k\leq N}^{N}\pared\E(h)\cap \pared \E(k).
\end{align*}
Usually the singular set of an $N$-cluster $\E$ is defined just as $(\pa \E \setminus \pared \E)\cap \Om$ and all the regularity results for these kind of objects involve this definition of singular sets. The stronger regularity of the chambers given by Theorem \ref{mainthm1} somehow affect the behavior of the singular set.  For example consider the case $n\leq 7$. In this case, according to Theorem \ref{mainthm1}, for a Cheeger $N$-cluster it must hold that 
$$(\pa \E \setminus \pared \E)\cap \Om=\emptyset,$$
and this would lead us to say that the singular set of a Cheeger $N$-cluster is empty which is clearly not the case. Indeed let us highlights that there is somehow an "hidden chamber" that plays a key role and influences the behavior of the global structure of these objects, namely the \textit{external chamber}:
$$\E(0)=\Om\setminus \left(\bigcup_{i=1}^N \E(i)\right).$$

Even if Theorem \ref{mainthm1} provides a satisfactory description of $\Sigma(\E(i),\Omega)$, this does not exhaust the analysis of the singular set of $\E$. Indeed the chamber $\E(0)$ is not regular after all and there are points in $\pa \E(0)$ of cuspidal type. For a complete description of the singularity, the correct definition of \textit{singular set of a Cheeger $N$-cluster $\E$ in the Borel set $F$} must be given as
\begin{align}
\S(\E;F)&:=\S(\E(0);F) \cup \bigcup_{i=1}^N \S(\E(i);F),\label{insieme singolare 1}
\end{align}
where for $i\neq 0$ the set $\S(\E(i))$ are the ones defined in \eqref{insieme singolare i} and \eqref{insieme singolare ii}, while for $i=0$ we clearly set
	\[
	\S(\E(0);F)=[\pa \E(0)\setminus \pared\E(0) ]\cap F.
	\]
With these definitions, $(\pa \E\setminus \pared \E)\cap \Om \subseteq \S(\E;\Om)$. Since Theorem \ref{mainthm1} do not provides information about $\S(\E(0))$, we focus our attention on it in Subsection \ref{cpt 4 sbsct Proof of mainthm2} where the following theorem is proved. 
\begin{theorem}\label{mainthm2}
Let $1\leq n\leq 7, N\geq 2$, Let $\Om\subset \R^n$ be an open, connected, bounded set with $C^1$ boundary and finite perimeter and $\E$ be a Cheeger $N$-cluster of $\Om$. Then the following statements hold true.
\begin{itemize}
\item[(i)] $\E(0)$ is not empty and $\H^{n-1}(\pa\E(0)\cap \pa\E(j))>0$ for all $i=1,\ldots,N$;
\item[(ii)]  $\S(\E(0);\Om)=\pa\E(0)\cap \E(0)^{\zero}$, $\S(\E(0);\Om)$ is closed and 
\begin{eqnarray}
\S(\E(0);\Om)  &=&\Om \cap \bigcup_{\substack{j,k=1, \\ k\neq j}}^N  (\pa \E(j)\cap \pa \E(k) \cap \pa \E(0) ) 
\end{eqnarray}
\end{itemize}
\end{theorem}
Here $\E(0)^{(0)}$ denotes the collection of the \textit{zero $n$-dimensional density points of $\E(0)$}. In general in the sequel we are adopting the notation $E^{(t)}$ by meaning \textit{the collection of points of $\R^n$ where the $n$-dimensional density of the set $E$ exists and it is equal to $t$.}

\begin{remark}\label{controesempioconne}
\rm{
Note that Assertion $(ii)$ of Theorem \ref{mainthm2}, stated as above, would be meaningless if we do not ensure that $|\E(0)|>0$ (that is Assertion $(i)$, proved in Proposition \ref{ognicameraconfinaconilvuoto}). The assumption on $\Om$ to be connected and with $C^1$-boundary are the necessary ones to ensure the validity of this fact. Probably, the theorem remains true also by replacing \textit{$C^1$ boundary} with \textit{Lipschitz boundary}. Anyway we prefer to state and prove it by taking advantage of this stronger regularity on $\pa \Om$ in order to avoid some technicality. Let us also point out that there are situations where $\Om$ is not connected or $\pa \Om$ is not Lipschitz and where $\E(0)$ turns out to be empty. For example, given a set $\Om$ and one of its Cheeger $N$-cluster $\E$, we provide a counterexample by defining the new open set 
$$\Om_0=\left(\bigcup_{j=1}^N \mathring{\E(j)}\right).$$
The $N$-cluster $\E$ will be a Cheeger $N$-clusters of $\Om_0$ also and, by construction, $|\E(0)|=0$ (see Figure \ref{controesempio}).  The reason is that $\Om_0$ has no regular boundary.  As a further example one may also consider the case when $\Om$ is the union of $N$ disjoint balls. Anyway, it is reasonable to expect that, no matter what kind of ambient space $\Om$ we choose, for $N$ sufficiently large the chamber $\E(0)$ will be not empty.}
\end{remark}
\begin{figure}
\begin{center}
 \includegraphics[scale=1]{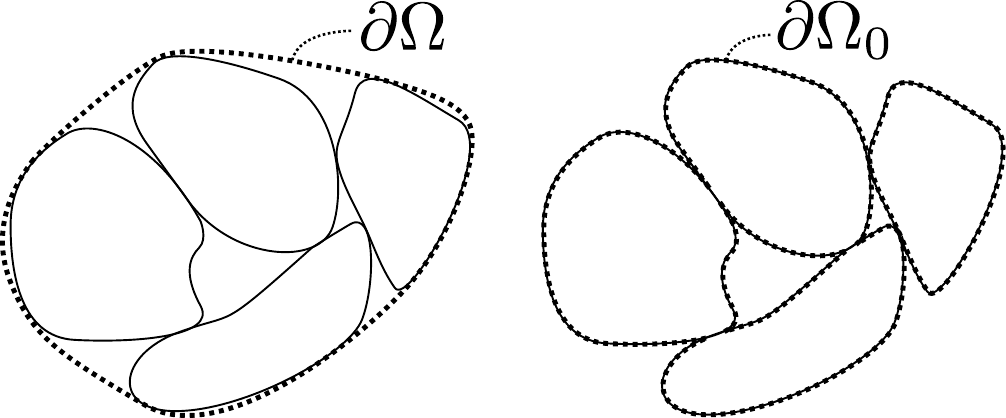}\caption{{\small The set $\Om_0$ built as the union of the interior of 
 the Cheeger $N$-cluster of an open set $\Om$. The external chamber of this Cheeger $N$-cluster of $\Om_0$ is empty because of the 
 cusps at the boundary of the open set.}}\label{controesempio}
 \end{center}
\end{figure}
\begin{remark}
\rm{Note that we ask for the dimension $n$ to be less than $7$. That is because, to prove Theorem \ref{mainthm2}, we exploit the regularity given by \ref{mainthm1} and we prefer to deal with the favorable case $n\leq 7$ where the singular set $\S(\E(i);\Om)=\emptyset$ for $i\neq 0$. Let us also point out that Assertion $(ii)$ remains true also in dimension bigger than $7$ up to replace $\Om$ with $\Om_0=\Om\setminus \cup_{i\neq 0} \S(\E(i);\Om)$. The interesting and not-trivial fact is that we actually do not know if assertion (i) remains true in dimension bigger than $7$ since, in the proof of Proposition \ref{ognicameraconfinaconilvuoto} (the crucial one in order to prove assertion $(i)$), we make a strong use of the fact $\S(\E(i);\Om)=\emptyset$. Roughly speaking in dimension bigger than $7$  it could happen that the chambers, by taking advantage of the possible presence of singular points $x\in \S(\E(i);\Om)$, can be combined in a way that kill $\E(0)$ even under a strong regularity assumption on $\Om$.}
\end{remark}

\begin{remark}\label{the singular set}
\rm{Somehow assertion $(ii)$ of Theorem \ref{mainthm2} is saying that the only singular points of $\E$ are the one where a cusp is attained.  Now we can give a complete description of the singular set $\S(\E;\Om)$ of a Cheeger $N$-cluster of an open, bounded, connected set $\Om$ with finite perimeter and $C^1$ boundary in dimension less than or equal to $7$. By combining Assertion $(iii)$ in Theorem \ref{mainthm1} and assertion $(i)$ in \ref{mainthm2} we can write
$$\S(\E;\Om)=\S(\E(0);\Om)=(\pa \E(0)\cap \E(0)^{(0)})\cap \Om.$$
}
\end{remark}

\subsection{The planar case}
Theorem \ref{mainthm2} gives us a precise structure of $\S(\E;\Om)$. We do not focus here on the singular set $\S(\E;\pa\Om)$ anyway, by exploiting the $C^1$-regularity assumption on $\pa \Om$, it is possible to prove a result in the spirit of Theorem \ref{mainthm2} also for the singular set $\S(\E(0);\pa \Om)$ (and thus characterize $\S(\E;\pa\Om)$). Let us point out that, at the present, the crucial information $\H^{n-1}(\S(\E(0);\Om))=0$ is missing. We are able to fill this gap when the ambient space dimension is $n=2$, together with some remarkable facts stated in the following theorem (proved in Subsection \ref{cpt 4 sbsct proof of mainthm3}).  
\begin{theorem}\label{mainthm3}
Let $n=2, N\geq 2$. Let $\Om\subset \R^2$ be an open, connected, bounded set with $C^1$ boundary and finite perimeter and $\E$ be a Cheeger $N$-cluster of $\Om$. Then the following statements hold true.
\begin{itemize}
\item[(i)] The singular set $\S(\E(0);\Om)$ is a finite union of points $\{x_j\}_{j=1}^k \subset \Om$. 
\item[(ii)] For every $j,k=0,\ldots,N$, $k \neq j$ the set
$$E_{j,k}:=[\pa \E(j)\cap \pa \E(k) \cap\Om]\setminus \S(\E(0);\Om)$$ 
is relatively open in $\pa \E(j)$ ($\pa \E(k)$) and is the finite union of segments and circular arcs. Moreover the set $\E(j)$ has constant curvature $C_{j,k}$ inside each open set $A$ such that $A\cap \pa \E(j) \subseteq E_{j,k}$. The constant $C_{j,k}$ is equal to:
\begin{equation}
C_{j,k}= \left\{
\begin{array}{cc}
\frac{|\E(k)|h(\E(j))-|\E(j)|h(\E(k))}{|\E(j)|+|\E(k)|}, & \text{if $k\ne 0$},\\
 & \\
h(\E(j)), & \text{if $k=0$}.
\end{array}
\right.
\end{equation}
As a consequence the set $\E(k)$ has constant curvature $C_{j,k}=-C_{k,j}$ inside each open set $A$ such that $A\cap \pa\E(k) \subseteq E_{k,j}$ ($=E_{j,k}$);
\item[(iii)] Each chamber $\E(j)\cc \Om$ is indecomposable.
\end{itemize} 
\end{theorem}
\begin{figure}
\centering
 \includegraphics[scale=3.0]{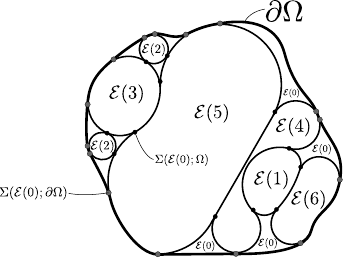}\caption{{\small An example of a possible Cheeger $6$-cluster in dimension $n=2$ suggested by Theorems \ref{mainthm1}, \ref{mainthm2} and \ref{mainthm3}}} \label{ese}
\end{figure}

We refer the reader to \cite[Section 17.3]{maggibook} where the Definition of distributional mean curvature of a finite perimeter set $E$ together with a satisfying treatment of this concept is provided.

\begin{remark}
\rm{
Theorems \ref{mainthm1}, \ref{mainthm2} and \ref{mainthm3} allow us to provide examples of planar Cheeger $N$-cluster. The one depicted in Figure \ref{ese} is a possible Cheeger $6$-clusters. Let us highlight that we do not want to suggest that the object in the figure is exactly the Cheeger $6$-cluster of the set $\Om$. We just want to point out the possible structure of such objects.
}
\end{remark}
\begin{remark}\label{remark controfinitezza}
\rm{Let us notice that Assertion $(i)$ of Theorem \ref{mainthm3} could fail when we replace  $\S(\E(0);\Om)$ with $\S(\E(0);\pa \Om)$. Indeed we can always modify $\Om$ at the boundary in order to produce a set $\Om_0$ having the same Cheeger $N$-clusters of $\Om$ and kissing the boundary of some $\pa \E(i)$ in a countable number of points (see Figure \ref{controfinitezza}). 
\begin{figure}
\begin{center}
 \includegraphics[scale=0.8]{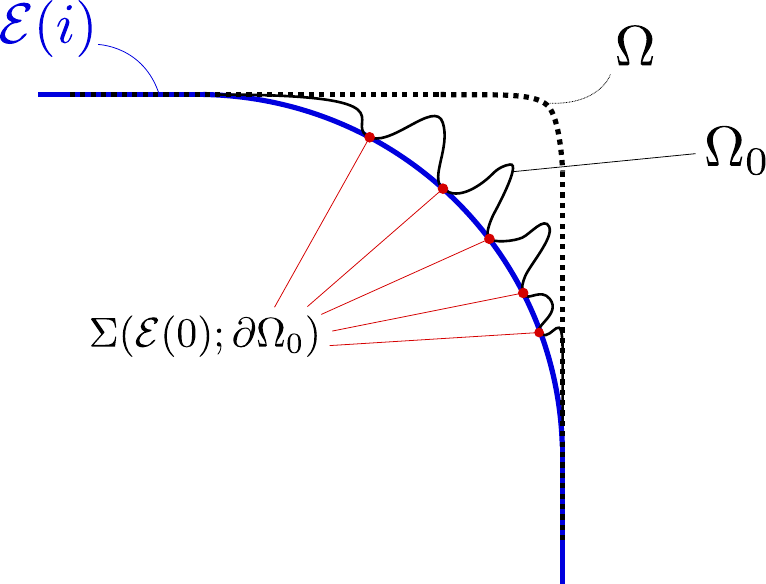}\caption{{\small By gently pushing $\pa \Om$ we can build as many contact points as we want. This proves that Assertion $(i)$ in Theorem \ref{struttura singolaritapiano} does not hold in general for $\S(\E(0);\pa \Om)$.} }\label{controfinitezza}
\end{center}
\end{figure}
}
\end{remark}

\begin{remark}
\rm{We speak of ``curvature of chambers" $\E(j),\E(k)$, instead of curvature of interfaces $\pa\E(j)\cap \pa\E(k)$ in order to point out that the sign of the constant $C_{j,k}$ depends on whether we are looking at $\pa\E(j)\cap \pa\E(k)$ as a piece of the boundary of $\E(j)$ or as a piece of the boundary of $\E(k)$ (namely it depends on the direction of the unit-normal vector to $\pa \E(j)\cap \pa\E(k)$ that we choose).
}
\end{remark}

\begin{remark}
\rm{Note that the set $E_{j,k}$ could be empty. For example, if $\pa\E(j) \cap \pa\E(k)\cap \Om=\{x\}$ consists of a single point, thanks to our characterization (assertion $(ii)$ Theorem \ref{mainthm2}) $x\in \S(\E(0);\Om)$. However, for some $k=0,\ldots, N$, $k\neq j$ it must clearly holds $\H^1 (E_{j,k})>0$. The natural question is whether there exists a chamber $\E(j)$ such that $E_{j,k}=\emptyset$ for all $k\neq 0$. We provide a lemma (Lemma \ref{nel piano!}) that excludes this possibility whenever $\E(j) \cc \Om$ and this will be our starting point for proving assertion $(iii)$ in Theorem \ref{mainthm3}. 
}
\end{remark}

\begin{remark}
\rm{ Since $E_{j,k}$ is relatively open in $\pa \E(j)\cap \pa\E(k)$ we can find an open set $A$ such that $A\cap \pa\E(k)\cap \pa \E(j)=E_{j,k}$ and conclude that $\pa \E(j)$ must have constant mean curvature in $A$ (that is, on $E_{j,k}$). In the sequel we sometimes refer to the distributional mean curvature of $\E(i)$ on $E\subset\pa \E(i)$, a relatively open subset of $\pa\E(i)$,  as the distributional mean curvature of $\E(i)$ inside the open set $A$ such that $A\cap \pa \E(i)=E$. 
}
\end{remark}

\begin{remark}
\rm{Assertions $(i)$ and $(ii)$ in Theorem \ref{mainthm3} tell us that a chamber $\E(j)$ has distributional curvature inside $\Om$ equal to
$$H_{\E(j)}(x)=\sum_{\substack{k=0\\ k,\neq j}}^{N} C_{j,k} \ca_{E_{j,k}}(x), \ \ \ \ \ \text{for $\H^1$-almost every $x\in \pa\E(j)\cap\Om$ }.$$ 
Indeed, since the set $\S(\E(0);\Om)$ is finite, $\pa\E(j)\cap \Om$ is $\H^1$-equivalent to $\bigcup_{k\neq j} E_{j,k}$. In particular, if $T\in C_c^{\infty}(\Om;\R^2)$ then
\begin{align*}
\int_{\pa \E(j)\cap \Om }\dive_{\E(j)}(T)\d \H^{1}(x)&= \sum_{\substack{k=0\\ k,\neq j} }^N \int_{\pa \E(j) \cap E_{j,k}\cap \Om }\dive_{\E(j)}(T)\d \H^{1}(x)\\
&=\sum_{\substack{k=0\\ k,\neq j}}^N \int_{\pa \E(j) \cap E_{j,k}\cap \Om }C_{j,k} (T\cdot \nu_{\E(j)})(x)\d \H^{1}(x)\\
&= \int_{\pa \E(j) \cap \Om } (T \cdot \nu_{\E(j)} )(x) \sum_{\substack{k=0\\ k,\neq j} }^N C_{j,k}\ca_{{E}_{j,k}}(x) \d \H^{1}(x).
\end{align*}
}
\end{remark}

We finally point out that, even if the indecomposability of the chambers is usually an hard task in the tessellation problems, in this case, thanks to a general fact for Cheeger sets (Proposition \ref{nel piano!}), we can easily achieve the proof of Assertion $(iii)$ in Theorem \ref{mainthm3}. This will be particularly useful when focusing our attention on the asymptotic behavior of $H_N$.

\section{Existence and regularity}\label{cpt 4 sct Existence and regularity}

This section is devoted to the proof of Theorem \ref{mainthm1}. We premise the following subsection where some technical facts are recalled.
\subsection{Technical tools}

\begin{definition}[$(\Lambda,r_0)$-perimeter-minimizing inside $\Om$, \cite{maggibook} pp. 278-279]\label{Lambdarminimi}
We say that a set of finite perimeter $E$ is a $(\Lambda,r_0)$-perimeter-minimizing in 
$\Om$ if for every $B_r\subset \Om$ with $r<r_0$ and every set $F$ such that $E\Delta F\subset\subset B_r$, it holds
$$P(E;B_r)\leq P(F;B_r)+\Lambda |E\Delta F|.$$
\end{definition}
The following theorem clarifies why Definition \ref{Lambdarminimi} is so important (see \cite[Chapter 21 and pp. 354, 363-365]{maggibook}).
\begin{theorem}\label{regularity} 
If $\Om$ is an open set in $\R^n$, $n\geq 2$ and $E$ is a $(\Lambda,r_0)$-perimeter-minimizing in $\Om$, with 
$\Lambda r_0\leq 1$, then for every $\a\in(0,1)$ the set $\Om\cap \partial^*E$ is a $C^{1,\a}$ hypersurface that is relatively open in $\Om\cap \partial E$, and it is $\H^{n-1}$ equivalent to $\Om \cap \partial E$. Moreover, setting 
$$\Sigma(E;\Om):=\Om\cap (\partial E\setminus \partial^*E),$$
then the following statements hold true:
\begin{itemize}
 \item[(i)] if $2\leq n\leq 7$, then $\Sigma(E;\Om)$ is empty;
 \item[(ii)] if $n=8$, then $\Sigma(E;\Om)$ is discrete;
 \item[(iii)] if $n\geq 9$, then $\H^s(\Sigma(E;\Om))=0$ for every $s>n-8$. 
\end{itemize}
\end{theorem}
In every dimension greater than or equal to $8$ it is possible to exhibit an example of a $(\La, r_0)$-perimeter-minimizing set $E$ with $\H^{n-8}(\S(E))>0$ (see \cite{de2009short}, \cite{bombieri1969minimal}, \cite[Section 28.6]{maggibook}). Hence assertion $(iii)$ cannot be improved and the only thing that we can say is that the Hausdorff dimension of the singular set $\S(E;\Om)$ is at most $n-8$: $\dim(\S(E))\leq n-8$.\\

Before entering in the details of the proof of Theorem \ref{regolare} let us remark that, by exploiting the general structure of sets of finite perimeter (see \cite[Theorems 15.5, 15.9, 16.2, 16.3]{maggibook}), it is possible to derive the inequality
\begin{equation}\label{diffe}
P(F\setminus E; A )+P(E\setminus F;A)\leq P(F;A)+P(E;A)
\end{equation}
holding for every couple of sets $E,F$ of locally finite perimeter and for every open set $A$. \\

In order to prove Theorem \ref{regolare} we also recall the following definition. We say that a set of finite perimeter $M$ has \textit{distributional mean curvature less than $g\in L^1_{loc}(\Om)$ in $\Om$} if, there exists $r_0$ such that for every $B_r\cc \Om$ with $r<r_0$ and for every $L\subseteq M$ with $M\setminus L \cc B_r$, it holds
\begin{equation}
P(M;B_r)\leq P(L;B_r)+\int_{M\setminus L} g(x)\d x.
\end{equation}

We also premise the following technical lemma. Since we were not able to find in literature a proof of this fact we provide also a proof.

\begin{lemma}\label{tecnico}
If $E_1,\ldots,E_k$ are $k$ sets of locally finite perimeter such that 
$$|E_i\cap E_j|=0 \ \ \ \ \forall \ i\neq j,$$
then the following holds:
\begin{equation}\label{spaghetto}
\begin{split}
\pa^* \left(\bigcup_{i=1}^k E_i\right)&\approx \left( \bigcup_{i=1}^k \pa^*E_i\right) \setminus \left(\bigcup_{\substack{ i,j=1 \\ j\neq i}}^k \pared E_j\cap \pared E_i\right)\\
&=\left[ \bigcup_{i=1}^k \pa^*E_i \setminus \left(\bigcup_{\substack{ j=1 \\ j\neq i}}^k \pared E_j\cap \pared E_i\right)\right]
\end{split}
\end{equation}
where the symbol $\approx$ means \textit{equal up to an $\H^{n-1}$-negligible set}. In particular for every ball $B_r=B_r(x)$ it holds:
\begin{equation}\label{peri N-cluster}
P\left(\bigcup_{i=1}^k E_i ;B_r\right)=\sum_{i=1}^k P(E_i;B_r)-\sum_{\substack{i,j=1, \\ j\neq i}}^k \H^{n-1}(\pa^* E_i\cap \pa^*E_j \cap B_r). 
\end{equation}
\end{lemma}
\begin{proof}
Relation \eqref{peri N-cluster} follows straightforwardly from \eqref{spaghetto}.
We recall from Federer's Theorem \cite[Theorem 16.2]{maggibook} that $\pared E\approx E^{\left(\frac{1}{2}\right)}$ for every locally finite perimeter set $E$. Hence, by setting
$E_0=\bigcup_{i=1}^k E_i,$
it is enough to prove that there exist two $\H^{n-1}$-negligible set $M_1,M_2$ such that
\begin{equation}\label{quello che vorrei dire}
E_0^{\mez} \subseteq  M_1\cup\left[  \left(\bigcup_{i=1}^k E_i^{\left(\frac{1}{2}\right)}\right) \setminus\left(\bigcup_{\substack{i,j=1\\ j\neq i} }^k E_i^{\mez}\cap E_j^{\mez}\right)\right]\subseteq (E_0^{\mez}\cup M_2).
\end{equation}
Let us also recall that, if $E$ is a set of locally finite perimeter then there exists an $\H^{n-1}$-negligible set $R$ with the following property
$$\R^n=E^{\zero}\cup E^{\mez}\cup E^{\uno}\cup R.$$
Thus, for every $i=0,\ldots,k$, we choose $R_i$ to be the $\H^{n-1}$-negligible set such that
	\begin{equation}\label{STAR}
	 \R^n=E_i^{\zero}\cup E_i^{\mez}\cup E_i^{\uno}\cup R_i,
	 \end{equation}	
and we set 
$$M_1:=\left(E_0^{\mez}\cap \bigcup_{i=1}^k R_i\right), \ \ \ M_2:= \bigcup_{i=1}^k R_i.$$
We prove that \eqref{quello che vorrei dire} holds with this choice of $M_1,M_2$ (note that $\H^{n-1}(M_1)=\H^{n-1}(M_2)=0$ is immediate). Let us set, for the sake of brevity
$$F:=M_1\cup \left[ \left(\bigcup_{i=1}^k E_i^{\left(\frac{1}{2}\right)}\right) \setminus\left(\bigcup_{\substack{i,j=1\\ j\neq i} }^k E_i^{\mez}\cap E_j^{\mez}\right)\right],$$
and divide the proof in two steps.\\

\textit{Step one: $E_0^{\mez}\subseteq F$}. In particular we prove that if $x\notin F$ then $x\notin E_0^{\mez}$. For $x\notin F$ one of the following must be in force
\begin{itemize}
		\item[a)] $\displaystyle x\notin M_1 \text{\ and \ }x\in \left(\bigcup_{i=1}^k E_i^{\left(\frac{1}{2}\right)}\right)\cap \left(\bigcup_{\substack{i,j=1\\ j\neq i} }^k E_i^{\mez}\cap E_j^{\mez}\right)$.
		\item[b)] $\displaystyle x\notin M_1 \text{\ and \ } x\notin  \left(\bigcup_{i=1}^k E_i^{\left(\frac{1}{2}\right)}\right) $ and in this case either:
             		 \begin{itemize}
							\item[b.1)] $\displaystyle x \notin E_0^{\mez}  \text{\ and \ } x\in \bigcup_{i=1}^kR_i \text{\ and \ } x\notin  \left(\bigcup_{i=1}^k E_i^{\left(\frac{1}{2}\right)}\right) $;
							\item[b.2)] $\displaystyle x \notin E_0^{\mez}  \text{\ and \ } x\notin \bigcup_{i=1}^kR_i \text{\ and \ } x\notin  \left(\bigcup_{i=1}^k E_i^{\left(\frac{1}{2}\right)}\right) $;
							\item[b.3)] $\displaystyle x \in E_0^{\mez}  \text{\ and \ } x\notin \bigcup_{i=1}^kR_i \text{\ and \ } x\notin  \left(\bigcup_{i=1}^k E_i^{\left(\frac{1}{2}\right)}\right) $.
						\end{itemize}
\end{itemize}
If situation a) is in force we immediately have that $x\in E_i^{\mez}\cap E_j^{\mez}$ for some $i\neq j$ which leads to $x\in E_0^{\uno}$ (since $|E_i\cap E_j|=0$) and thus $x\notin E_0^{\mez}$. Since b.1) and b.2) implies straightforwardly $x\notin E_0^{\mez}$, we need just to verify that situation b.3) cannot be attained. Assume b.3) is in force and note that, for every $i=1,\ldots,k$, thanks to \eqref{STAR} it must hold $x\in E_{i}^{\uno}\cup E_{i}^{\zero}$. If $x\in  E_{i}^{\zero}$ for all $i$ we have $x\in E_0^{\zero}$. If, instead, $x\in E_i^{\uno}$ for some $i$ then $x\in E_0^{\uno}$. In both cases we reach a contradiction because of $x\in E_0^{\mez}$.\\

\textit{Step two: $F\subseteq (E_0^{\mez}\cup M_2)$}. For every $x\in F$ one of the following must be in force.
\begin{itemize}
\item[a)] $\displaystyle x\in M_1$;
\item[b)] $\displaystyle x\in\left(\bigcup_{i=1}^k E_i^{\left(\frac{1}{2}\right)}\right) \setminus\left(\bigcup_{\substack{i,j=1\\ j\neq i} }^k E_i^{\mez}\cap E_j^{\mez}\right) \text{\ and \ }  \displaystyle x\notin  M_1$ and in this case either:
             		 \begin{itemize}
							\item[b.1)]$\displaystyle x\in\left(\bigcup_{i=1}^k E_i^{\left(\frac{1}{2}\right)}\right) \setminus\left(\bigcup_{\substack{i,j=1\\ j\neq i} }^k E_i^{\mez}\cap E_j^{\mez}\right) \text{\ and \ }  \displaystyle x\notin \bigcup_{i=1}^k R_i$;
							\item[b.2)] $\displaystyle x\in \left(\bigcup_{i=1}^k E_i^{\left(\frac{1}{2}\right)}\right) \setminus\left(\bigcup_{\substack{i,j=1\\ j\neq i} }^k E_i^{\mez}\cap E_j^{\mez}\right) \text{\ and \ }  \displaystyle x\notin E_0^{\mez}$;
					\end{itemize}
\end{itemize}
If a) is the case, then $x\in M_1\subset E_0^{\mez}$ and we are done. If b.1) is in force then there exists exactly one $j$ such that $x\in E_j^{\mez}$ and $x\in E_{i}^{\zero} $ for $i\neq 0, j$ since the sets $\{E_h\}_{h=1}^k$ are disjoint up to an $\L^n$-negligible set. Thus
\begin{align*}
\frac{|(\R^n\setminus E_0) \cap B_r(x)|}{\om_n r^n}&=1-\frac{|E_j \cap B_r(x)|}{\om_n r^n}-\sum_{\substack{i=1, \\ i\neq j}}^k \frac{|E_i \cap B_r(x)|}{\om_n r^n},
\end{align*}
which, passing to the limit as $r$ goes to $0^+$ implies $x\in (\R^n\setminus E_0)^{\mez}=E_0^{\mez}$. Finally, by considering situation b.2) we deduce that there exists exactly one $j\in \{1,\ldots,k\}$ such that $x\in E_{j}^{\mez}$ and $x\in E_i^{\zero}\cup R_i$ for $i\neq j$. If $x\in E_{i}^{(0)}$ for all $i\neq 0, j$ then, as above $x\in E_0^{\mez}$ and this is a contradiction (in this situation we are assuming $x\notin E_0^{\mez}$). Hence there is an index $i\neq 0$ such that $x\in R_i$ which means $x\in M_2$. The proof is complete.
\end{proof}

\subsection{Proof of Theorem \ref{mainthm1}}
We start by proving the existence and then, separately, we prove the regularity for Cheeger $N$-clusters.

\begin{theorem}[Existence of Cheeger $N$-clusters.]\label{existence}
Let $\Omega$ be a bounded set with finite perimeter and $0<|\Om|$. For every $N\in \N$ there exists 
a Cheeger $N$-cluster of $\Om$, i.e. an $N$-cluster $\E\subseteq \Om$ such that:
$$H_N(\Om)=\sum_{i=1}^N\frac{P(\E(i))}{|\E(i)|}.$$
Moreover each Cheeger $N$-cluster of $\Om$ has the following properties:
\begin{equation}\label{nontrivialchambers}
|\E(i)|\geq \frac{n^n\om_n}{2^n H_N(\Omega)^n}  \ \ \ \text{for all } \ i=1,\ldots,N,
\end{equation}
\begin{equation}\label{selfcheeger}
 h(\E(i))=\frac{P(\E(i))}{|\E(i)|} \ \ \ \text{for all } \ i=1,\ldots,N.
\end{equation}
\end{theorem}
\begin{proof}
Clearly $H_N(\Omega)<+\infty$ since we can always choose, for example, $B_1,\ldots B_N$ disjoint balls
such that $|B_i\cap\Omega|>0$ and obtain
\begin{equation}\label{notinfinity}
H_N(\Omega)\leq \sum_{i=1}^N \frac{P( B_i \cap \Omega)}{|B_i\cap \Omega|}<+\infty \, .
\end{equation}
Moreover, thanks to the fact that $\Om$ is bounded we deduce $H_N(\Om)>0$. Indeed for every $N$-cluster $\E\subseteq \Om$, the isoperimetric inequality for sets of finite perimeter implies
$$\sum_{i=1}^N\frac{P(\E(i))}{|\E(i)|}\geq nN\left(\frac{\om_n}{|\Om|}\right)^{1/n}$$
hence 
$$H_N(\Om)\geq nN\left(\frac{\om_n}{|\Om|}\right)^{1/n}>0.$$
Consider a minimizing sequence $\E^k=\{\E^k(i)\}_{i=1}^N$ of $N$-clusters such that
$$\lim_{k\rightarrow +\infty} \sum_{i=1}^N \frac{P(\E^k(i))}{|\E^k(i)|}=H_N(\Om).$$
Note that
\begin{align*}
 P(\E^k(i))&\leq|\Omega|\sum_{j=1}^N\frac{P(\E^k(i))}{|\E^k(i)}\leq 2|\Omega|H_N(\Omega).
 \end{align*}
Moreover, by exploiting again the isoperimetric inequality for sets of finite perimeter, we provide the bound
\begin{align*}
n\left(\frac{\om_n}{|\E^k(i)|}\right)^{\frac{1}{n}} &\leq \frac{P(\E^k(i))}{|\E^k(i)|}\leq 2H_N(\Omega)
\end{align*}
and thus
\begin{eqnarray}
\sup_{k}\left\{\max_i \left\{P(\E^k(i))\right\}\right\}&\leq &2|\Omega|H_N(\Omega), \label{a} \\
\inf_{k}\left\{\min_i\left\{|\E^k(i)|\right\}\right\}&\geq & \frac{n^n\om_n}{2^n H_N(\Omega)^n}. \label{b} 
\end{eqnarray}
Thanks to the boundedness of $\Omega$ and to \eqref{a}, we can apply the compactness theorem 
for sets of finite perimeter (see \cite[Theorem 12.26]{maggibook}) and deduce that, up to a subsequence, each sequence of chambers $\E^k(i)$ is converging in $L^1$ to some 
$\E(i)\subseteq \Omega$ as $k\rightarrow +\infty$. Equation \eqref{b} implies the lower bound \eqref{nontrivialchambers} while 
the lower semicontinuity of distributional perimeter (see \cite[Proposition 12.15]{maggibook}) yields:
\begin{align*}
H_N(\Omega)&\leq \sum_{i=1}^N\frac{P(\E(i))}{|\E(i)|}\leq \sum_{i=1}^N \liminf_{k\rightarrow \infty} \frac{P(\E^k(i))}{|\E^k(i)|}\\
&\leq \liminf_{k\rightarrow+\infty}\sum_{i=1}^N \frac{P(\E(i)^k)}{|\E^k(i)|}=H_N(\Omega).
\end{align*}
Property \eqref{selfcheeger} immediately follows from minimality.
\end{proof}
\begin{remark}\rm{
Thanks to  property \eqref{selfcheeger} $H_N$ can be equivalently defined as 
\begin{equation}\label{N-cheeger constant 2}
H_N(\Om)=\left\{\sum_{i=1}^{N}h(\E(i)) \ \Big{|} \ \E\subseteq \Om \text{ N-Cluster}\right\}.
\end{equation}
}
\end{remark}

We now show that every Cheeger $N$-cluster of a given open set is a $(\La,r_0)$-perimeter-minimizing inside $\Om$ that will implies immediately assertion $(i),(iii),(iv),(v)$ in Theorem \ref{mainthm1} by applying the regularity Theorem \ref{regularity}. \\

Note that, for proving regularity in the case of Cheeger $N$-clusters we have to deal with the possible non trivial components $\pa \E(i)\cap \pa \E(j)$.  Roughly speaking, property \eqref{selfcheeger}, implies that both $\E(i)$ and $\E(j)$ must have mean curvature bounded from above. This leads us to say that the mean curvature of $\E(i)$ ($\E(j)$) on $\pa\E(i)\cap\E(j)$ must be bounded from below as well and so neither outer nor inner cusps can be attained. This approach is based on an idea from \cite{BaMa82}, where the authors prove a regularity result for the solutions of some obstacle problems. 
\begin{theorem}\label{regolare}
Let $\Om$ be an open bounded set and $\E$ be a Cheeger $N$-cluster of $\Om$. Then there exists $\Lambda,r_0>0$ depending on $\E$ 
with $\Lambda r_0\leq 1$, such that each $\E(i)$ is a $(\Lambda,r_0)$-perimeter-minimizing in $\Om$. As a 
consequence, for every $i=1,\ldots,N$ the set $\Om\cap\pared \E(i)$ has the regularity of Theorem \ref{regularity}.
\end{theorem}

\begin{proof}[Proof of theorem \ref{regolare}]
We start by fixing $i\in \{1,\ldots,N\}$ and by defining
$$M_i=\bigcup_{\substack{j=1, \\j\neq i }}^N\E(j).$$ 
We divide the proof in two steps.\\

\textit{Step one.} We prove that each $M_i$ has distributional mean curvature less than $H_N(\Om)$ in $\Om$. Let $B_r\subset\subset \Om$ be a ball and $L\subseteq M_i$ be a subset of finite perimeter of $M_i$ with $M_i\setminus L \subset\subset B_r$. What we need to prove is
\begin{equation}\label{limitiamolacurvatura}
P(M_i;B_r)\leq P(L;B_r)+H_N(\Om)|M_i\setminus L|. 
\end{equation}
Note that, up to choosing $r<r_0=\frac{n}{4H_N(\Om)}$ we can always assume $|\E(j)\cap L|>0$ for every $j\neq i$. Indeed $M_i\setminus L\cc B_r$ and, if by contradiction we assume $|\E(j)\cap L|=0$ for some $j\neq i$, this would mean $\E(j)\subset B_r$ up to a set of measure $0$ which implies (because of property \eqref{nontrivialchambers} and thanks to the choice of $r_0$) :
$$\frac{n^n\om_n}{2^nH_N(\Om)^n}<|\E(j)|<\om_n r^n<\frac{n^n\om_n}{4^nH_N(\Om)^n}$$
that is impossible.\\

By minimality it must hold:
\[
\frac{P(\E(j))}{|\E(j)|}\leq \frac{P(\E(j)\cap L)}{|\E(j)\cap L|} \ \ \ \ \text{ for every $j\neq i$,}
\]
that leads to:
\begin{align}
\frac{P(\E(j);B_r)+P(\E(j);B_r^c)}{|\E(j)|}&\leq \frac{P(\E(j)\cap L;B_r)+P(\E(j);B_r^c)}{|\E(j)|-|\E(j)\setminus L|}\nonumber\\
 P(\E(j);B_r)&\leq P(\E(j)\cap L;B_r)+|\E(j)\setminus L|h(\E(j)).\label{battle for falluja}
\end{align}
By exploiting \eqref{peri N-cluster} in Lemma \ref{tecnico} and \eqref{battle for falluja} above we obtain
\begin{align}
P(M_i;B_r)&=P\left(\cup_{j\neq i}\E(j);B_r\right)\nonumber\\
_{\text{\eqref{peri N-cluster} in Lemma \ref{tecnico}}}&=\sum_{j\neq i}P(\E(j);B_r)-\sum_{\substack{k,j\neq i,\ k\neq j}}\H^{n-1}(\partial^*\E(j)\cap \partial^* \E(k)\cap B_r)\nonumber\\
_{\eqref{battle for falluja}}&\leq \sum_{j\neq i}P(\E(j)\cap L;B_r) + |\E(j)\setminus L|h(\E(j))\nonumber\\
&-\sum_{\substack{k,j\neq i,\ k\neq j}}\H^{n-1}(\partial^*\E(j)\cap \partial^* \E(k)\cap B_r)\nonumber\\ 
&\leq \sum_{j\neq i}P(\E(j)\cap L;B_r) -\sum_{\substack{k,j\neq i,\ k\neq j}}\H^{n-1}(\partial^*\E(j)\cap \partial^* \E(k)\cap B_r)\nonumber\\
&+ H_N(\Om) |M_i\setminus L|  \label{il cacciatore},
\end{align} 
where in the last inequality we have used  the formulation of $H_N$ as in \eqref{N-cheeger constant 2}. By exploiting again Lemma \ref{tecnico} for $\{\E(j)\cap L\}_{j\neq i}$ we obtain
\begin{align}
P(M_i\cap L ; B_r)&=\sum_{j\neq i}P(\E(j)\cap L;B_r)\nonumber\\
&-\sum_{k,j\neq i \ \ k\neq j}\H^{n-1}(\pared (\E(j)\cap L) \cap \pared(\E(k)\cap L) \cap B_r).\label{full metal jacket}
\end{align}
Thanks to the fact that
\[
\pared (\E(j)\cap L) \cap \pared(\E(k)\cap L) \cap B_r \approx L^{(1)}\cap \pared\E(k)\cap\pared \E(j) \cap B_r
\]
we are lead to
\begin{equation}\label{platoon}
\sum_{j\neq i}P(\E(j)\cap L;B_r)-\sum_{k,j\neq i \ \ k\neq j}\H^{n-1}(\pared\E(k)\cap\pared \E(j) \cap B_r)\leq P(L; B_r),
\end{equation}
where we have exploited also $[(M_i\cap L)\Delta L] \cap B_r =\emptyset$. By combining \eqref{platoon} with \eqref{il cacciatore} we reach
\begin{align*}
P(M_i;B_r)&\leq P(L; B_r)+H_N(\Om) |M_i\setminus L|,
\end{align*} 
and we achieve the proof of Step one.\\

\textit{Step two.} We now prove that $\E(i)$ is a $(\Lambda,r_0)$-perimeter-minimizing for a suitable choice of $\Lambda$ 
and $r_0<\frac{n}{4H_N(\Om)}$ (according to Step one). Let $B_r\subset\subset \Om$ and $F$ be such that $F\Delta\E(i)\subset\subset B_r$. Define $E:=F\setminus M_i$ and observe, 
by minimality of $\E$ and by the relation $\E(i)\cap B_r^c=(F\setminus M_i)\cap B_r^c$, that:
\begin{align*}
\frac{P(\E(i))}{|\E(i)|} &\leq \frac{P(E)}{|E|}.\\
\end{align*}
Hence
\begin{align*}
\frac{P(\E(i);B_r)+P(\E(i);B_r^c)}{|\E(i)|} &\leq\frac{P(F\setminus M_i;B_r)+P(F\setminus M_i;B_r^c)}{|F|-|F \cap M_i|},\\
 &\leq\frac{P(F\setminus M_i;B_r)+P(\E(i);B_r^c)}{|\E(i)|+(|F\cap B_r|-|\E(i)\cap B_r|)-|F\cap M_i|}.
\end{align*}
If we expand the last inequality we get:
\begin{align*}
P(\E(i);B_r)|\E(i)|&\leq P(F\setminus M_i;B_r)|\E(i)|+P(\E(i))(|F\cap M_i|+|\E(i)\cap B_r|-|F\cap B_r|),
\end{align*}
which means (by observing that $F\cap M_i\subseteq F\setminus \E(i)$),
\begin{eqnarray}\label{quasifinita}
P(\E(i);B_r)&\leq& P(F\setminus M_i;B_r)+2h(\E(i))|\E(i)\Delta F|\, .
\end{eqnarray}
By making use of \eqref{diffe} we obtain
\begin{equation}\label{quasifinitachiave}
P(F\setminus M_i;B_r) \leq P(F;B_r)+P(M_i;B_r)-P(M_i\setminus F;B_r)
\end{equation}
Since $M_i\setminus F\subset M_i$ and $(M_i\setminus F)\Delta M_i\subset\subset B_r$ we can use step one (relation \eqref{limitiamolacurvatura}) with $L=M_i\setminus F$ for conclude that
\begin{eqnarray*}
P(M_i;B_r)&\leq & P(M_i\setminus F;B_r)+H_N(\Om)|M_i\setminus (M_i\setminus F)|\\
 &\leq & P(M_i\setminus F;B_r)+H_N(\Om)|M_i\cap F|.
 \end{eqnarray*}
 Hence
 \begin{equation}\label{e3}
 P(M_i;B_r)-P(M_i\setminus F;B_r)\leq  H_N(\Om)|F\setminus \E(i)|.
 \end{equation}
By plugging \eqref{e3} in \eqref{quasifinitachiave} we obtain
\begin{equation}\label{quasifinitachiave2}
P(F\setminus M_i;B_r)\leq P(F;B_r)+H_N(\Om)|\E(i)\Delta F|
\end{equation}
and by using \eqref{quasifinitachiave2} in \eqref{quasifinita} we find
$$P(\E(i);B_r)\leq P(F;B_r)+3H_N(\Om)|\E(i)\Delta F|.$$
By choosing $\Lambda=3H_N(\Om)$ and $r_0=\frac{1}{4 H_N(\Om)}$ we conclude that each $\E(i)$ is a 
$(\Lambda,r_0)$-perimeter-minimizing with $\Lambda r_0<1$ and we achieve the proof.
\end{proof}

Proof of assertion $(ii)$ can be viewed as a consequence of  \cite[Proposition 2.5, Assertion (vii)]{LP14} recalled below for the sake of clarity.
\begin{proposition}\label{leo}
Let $A$ be an open and bounded set and let $E$ be a Cheeger set $A$. Then
$$\pa^*A \cap \pa E \subseteq \partial^*E.$$
Moreover for every $x\in \pa^*A \cap \pa E$ 
it holds that 
$$\nu_{E}(x)=\nu_{A}(x),$$
where $\nu_{E}$, $\nu_{A}$ denotes the measure theoretic outer unit normal to $E$ and $A$ respectively.
\end{proposition}
The proof of Proposition \ref{leo} follows by combining the fact that each Cheeger set $E$ is a $(\Lambda,r_0)-$perimeter-minimizing in $A$ with the fact that the blow-ups of $\pa A$ at a point $x\in \pa^* A$ converge to an half plane.
\begin{proof}[Proof of Theorem \ref{mainthm1}]
Assertion $(i)$, $(iii),(iv), (v)$ follow by combining Theorems \ref{regolare} and \ref{regularity}. Assertion $(ii)$ is obtained by noticing that each chambers $\E(i)$ is a Cheeger set for 
	\[
	A=\Om \setminus \bigcup_{\substack{j=1,\\ j\neq i }}^N \E(j)
	\]
and then by applying Proposition \ref{leo}.
\end{proof}

\section{The singular set $\S(\E)$ of Cheeger $N$-clusters in low dimension}\label{cpt 4 sct The singular set of the Cheeger N-clusters in low dimension}
The following results are all stated and proved for open bounded and connected sets $\Om\subset \R^{n}$ having $C^{1}$ boundary and with the ambient space dimension less than $8$. We ask $\Om$ to be connected and with $C^1$ boundary because this is enough to avoid degenerate situations where $|\E(0)|=0$ (see Remark \ref{controesempioconne} where a Cheeger $N$-cluster with $|\E(0)|=0$ is provided).\\ 

We obtain the proof of Theorem \ref{mainthm2} by combining different results, sated and proved separately in Subsection \ref{cpt 4 sbsct Proof of mainthm2}. We premise some technical lemmas.

\subsection{Technical lemmas}

\begin{lemma}\label{complete regularity of cheeger N-clusters}
If  $n\leq 7$, $\Om$ is an open, bounded, connected sets with $C^1$ boundary and finite perimeter and $\E$ is a Cheeger $N$-cluster of $\Om$ it holds 
\[
\pared \E(i)=\pa \E(i) \ \ \ \ \ \text{for all $i\neq 0$}.
\]
 \end{lemma}
 \begin{proof}
 We decompose $\pa \E(i)$ as
 \[
 \pa\E(i)=(\pa\E(i)\cap \Om)\cup (\pa\E(i)\cap \pa\Om).
 \]
 and we note that
 \[
 \pa\E(i)\cap \Om =\pared \E(i)\cap \Om,
 \]
 because of Assertion (iii) of \ref{mainthm1}. Moreover, since $\Om$ has $C^1$ boundary $\pared \Om=\pa \Om$ and thus, thanks to Assertion (ii) we have also
 \[
 \pa \E(i)\cap \pa \Om = \pa \E(i)\cap \pared \Om \subseteq \pared \E(i).
 \]
 Hence
 \[
\pared\E(i)\subseteq \pa \E(i)=(\pa\E(i)\cap \Om)\cup (\pa \E(i)\cap \pa \Om)\subseteq \pared \E(i),
 \]
 and we achieve the proof.
 \end{proof}
 \begin{remark}\label{closed}
If $n\leq 7$, $\Om$ is an open, bounded, connected set with finite perimeter and $C^1$-boundary and $\E$ is a Cheeger $N$-cluster of $\Om$, by considering $\F(k)=\E(k)\cup \pa \E(k)$ for $k\neq 0$, thanks to Lemma \ref{complete regularity of cheeger N-clusters} we must have $|\F(k)\Delta \E(k)|\leq |\pa \E(k)|=0$ and thus
	\[
	P(\F(k))=P(\E(k)).
	\]
For this reason in the sequel we are always assuming that each chamber $\E(k)$ for $k\neq 0$ \textbf{is a closed set with $\pared \E(k)=\pa \E(k)$}.
 \end{remark}
\begin{lemma}\label{densita dei punti di omega}
Let $n\leq 7$, $N\geq 2$ and $\Om$ be an open, bounded and connected set with $C^{1}$ boundary and finite perimeter in $\R^{n}$. If $\E$ is a Cheeger $N$-cluster of $\Om$, then for every $x\in \R^{n}$ and every $k=1,\ldots,N$ there exists the $n$-dimensional density $\vt_n(x,\E(k))$ and it takes values:
\[
\vartheta_n(x,\E(k))=
\begin{sistema}
0 \ \ \text{if $x\notin \E(k)$};\\ 
\frac{1}{2}\ \  \text{if $x\in \pa\E(k)$};\\
1 \ \ \text{if $x\in \mathring{\E(k)}$}.
\end{sistema}
\]
\end{lemma}
\begin{proof} 
Each chamber $\E(k)$ for $k\neq 0$ is a closed set (see Remark \ref{closed}) and thus 
	\begin{align*}
		\E(k)^c&=\E(k)^{\zero},\\
		\mathring{\E(k)}&=\E(k)^{\uno}.
	\end{align*}
Lemma \ref{complete regularity of cheeger N-clusters} implies 
	\[
		\pa \E(k)=\pared \E(k)\subseteq \E(k)^{\mez}\subseteq \pa \E(k).
	\] 
\end{proof}

\subsection{Proof of theorem \ref{mainthm2}}\label{cpt 4 sbsct Proof of mainthm2}

We are now ready to prove two propositions that immediately imply Theorem \ref{mainthm2}. The following Proposition is needed in order to prove Assertion $(i)$ in Theorem \ref{mainthm2}.
\begin{proposition}\label{ognicameraconfinaconilvuoto}
Let $n\leq 7$, $N\geq 2$ and $\Om$ be an open, bounded and connected set with $C^{1}$ boundary and finite perimeter in $\R^{n}$. Let $\E$ be a Cheeger $N$-cluster of $\Om$. 
Then for every $i\in \{1,\ldots,N\}$ there exists $x\in \pa\E(i)$ such that $|B_s(x)\cap \E(0)|>0$ for all $s>0$.
\end{proposition}
\begin{proof}
Without loss of generality (and for the sake of clarity) we can assume $i=1$. We note that the proof of the lemma is a consequence of the following claim.\\

\textit{Claim.} $\pa \E(1)\setminus \left[\pa\Om \cup \bigcup_{k=2}^N \pa \E(k)\right]\neq \emptyset$.\\

Indeed, if the claim is in force then there exists $x\in \pa\E(1)\cap \Om$ and $x\notin \E(k)$ for all $k\neq 1$. Since the chambers are closed we can also find a small ball $B_s(x)\cc\Om$ such that $B_s(x)\cap \E(k)=\emptyset$ for all $k\neq 1$, implying (thanks to Lemma \ref{densita dei punti di omega})
\begin{align*}
|\E(0)\cap B_s(x)|&=|B_s|-\sum_{k=1}^n|\E(i)\cap B_s|=|B_s|-|\E(1)\cap B_s|>0
\end{align*}
(because $x\in \pa \E(1)=\E(1)^{\mez}$) and achieving the proof.\\

Let us focus on the proof of the claim. Thanks to the connectedness of $\Om$ it is easy to show that $\pa \E(1)\setminus \pa\Om \neq \emptyset$. If also $\pa \E(1)\cap \pa \E(k)=\emptyset$ for $k\neq 1$ the claim trivially holds. Otherwise it must exist at least an index $j\in\{2,\ldots,N\}$ such that $\pa\E(1)\cap \pa\E(j)\neq \emptyset.$ Assume without loss of generality $j=2$:
	\[
	\pa\E(1)\cap \pa\E(2)\neq \emptyset.
	\]
Choose $x\in \pa\E(1)\cap \pa\E(2)$ and let us denote by $M$ the connected component of $\pa \E(1)$ containing $x$. Note that $x\notin \pa \Om$. Otherwise we would have $x\in \pa\E(1)\cap \pa\E(2)\cap \pa\Om$ and because of the regularity of $\Om$ and thanks to Lemma \ref{densita dei punti di omega} this leads to a contradiction:
\begin{align*}
\frac{1}{2}&=\lim_{r\rightarrow 0^+} \frac{|\Om\cap B_r(x)|}{|B_r(x)|}=\lim_{r\rightarrow 0^+} \sum_{h=0}^N \frac{|\E(h)\cap B_r(x)|}{|B_r(x)|}\\
&\geq \lim_{r\rightarrow 0^+} \frac{|\E(1)\cap B_r(x)|}{|B_r(x)|}+\frac{|\E(2)\cap B_r(x)|}{|B_r(x)|}=1.
\end{align*}
Hence the following are in force:
\begin{equation}\label{leader maximo}
M\setminus \pa \Om\neq \emptyset, \ \ \ \ \ M\cap\pa\E(2)\neq \emptyset.
\end{equation}
We now note that, if 
\begin{equation}\label{non ne posso piu di questo lemma}
M\setminus \left[\pa \Om\cup \bigcup_{k=2}^N\pa \E(k)\right]=\emptyset
\end{equation} then, necessarily $M \subseteq\pa \E(2)\cap \Om$. Indeed considered 
	\[
	y\in \overline{M\cap \pa \E(2)}\cap \overline{(M\setminus \pa\E(2) )}=\bd_M(M\cap \pa \E(2)),
	\]
since \eqref{non ne posso piu di questo lemma} is in force (and since $y\in \bd_M(M\cap \pa \E(2))$) either there exists an index $k\neq 1,2$ such that $y\in \pa\E(k)$ or $y\in \pa \Om$. In both cases we reach a contradiction because $y$ would be a point of density $\frac{1}{2}$ for three disjoint sets ($\E(1),\E(2),\E(k)$ or $\E(1),\E(2),\Om^c $). Thus the only possibility is that $\bd_M(M\cap\pa\E(2))=\emptyset$ and since \eqref{leader maximo} is in force, by applying the following (topological) fact \eqref{topological fact} we conclude that $M=M\cap \pa\E(2)\subseteq\pa\E(2)$.
\begin{equation}\label{topological fact}
\begin{array}{c}
\text{If $M\subset \R^n$ is a closed connected set and $C\subseteq M$ is a non empty}\\
\text{subset of $M$, then $\bd_M(C):=\overline{C}\cap \overline{(M\setminus C)}=\emptyset$ if and only if $M=C.$}
\end{array}
\end{equation} 
As before $M\cap \pa\E(2)\cap \pa \Om=\emptyset$ otherwise we would have a point of density $\frac{1}{2}$ for three disjoint set $(\E(1),\E(2),\Om^c)$ and hence $M\subseteq \pa\E(2)\cap \Om$. This means that $M$ must be a closed $C^{1,\a}$ surface without boundary contained in $\Om$ and disjoint from the other sets $\E(k)$ and from $\pa \Om$, which means that one of the situation of Figure \ref{movimenti} has to be in force.
\begin{figure}
\centering
 \includegraphics[scale=0.8]{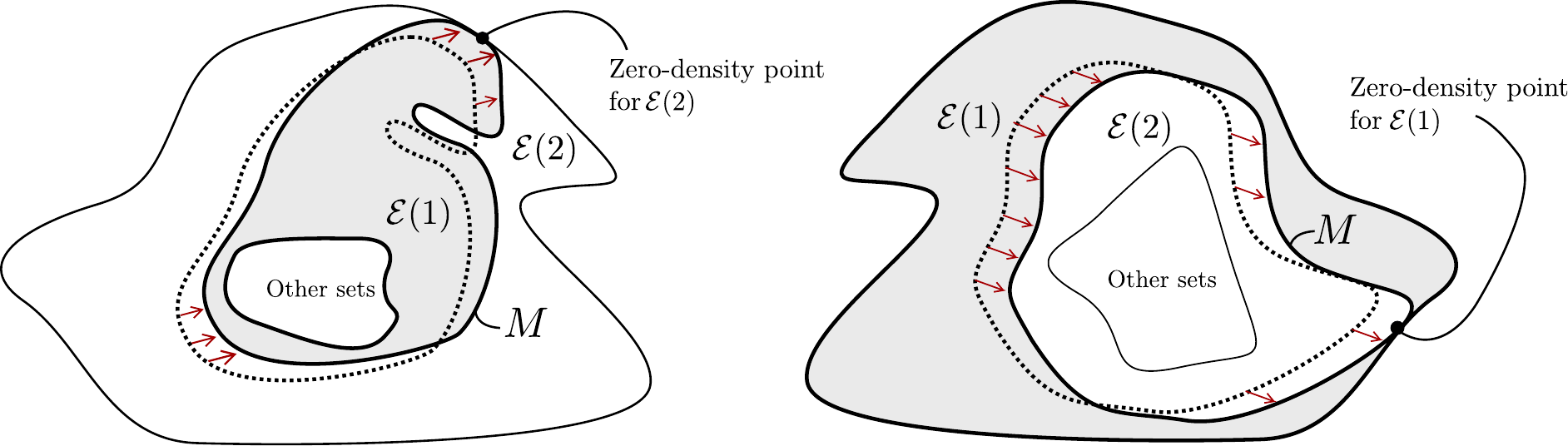}\caption{{\small If \eqref{non ne posso piu di questo lemma} holds, then one of these two situations must be in force and we can contradict regularity by simply translate $M$ until it kisses another part of the boundary yielding a not allowed point of density zero.}}\label{movimenti}
\end{figure}
We are thus able to move a little bit $M$, and whatever is bounded by $M$, inside $\Om$ as in Figure \ref{movimenti} until it kisses $\pa\E(2)$ or $\pa\E(1)$  (we easily exclude that $M$ bounds a hole of $\Om$ with a slight variation of this previous argument). In this way we produce a zero-density point for $\E(1)$ or for $\E(2)$ without changing $\sum_j\frac{P(\E(j))}{|\E(j)|}$ and this contradicts the regularity.\\

Hence \eqref{non ne posso piu di questo lemma} cannot holds and the claim is true.

\end{proof}

The next Proposition implies Assertion $(ii)$ in Theorem \ref{mainthm2}.
\begin{proposition}\label{prima caratterizazione}
Let $n\leq 7$, $N\geq 2$ and $\Om$ be an open, bounded and connected set with $C^{1}$ boundary and finite perimeter in $\R^{n}$. Let $\E$ be a Cheeger $N$-cluster of $\Om$. Then
\begin{align*}
\S(\E;\Om)&=\S(\E(0);\Om)\\
&=\{x\in \pa\E(0)\cap \Om \ | \ \vt_n(x,\E(0))=0\}\\
&=\Om\cap \bigcup_{\substack{i,j=1,\\ i\neq j}}^N \pa \E(i)\cap \pa\E(j)\cap \pa \E(0).
\end{align*}
\end{proposition}
\begin{proof}
Thanks to Proposition \ref{ognicameraconfinaconilvuoto} the set $\E(0)$ is not empty. As pointed out in Remark \ref{the singular set}, thanks to the regularity of each chambers, it is immediate that $\S(\E;\Om)=\S(\E(0);\Om)$. Let us denote (for the sake of brevity) by
	\begin{align*}
	\S_0&=\S(\E(0);\Om),\\
	A&=\{x\in \pa\E(0)\cap \Om \ | \ \vt_n(x,\E(0))=0 \}\\
	B&=\Om \cap \bigcup_{\substack{i,j=1,\\ k\neq j}}^N \pa \E(i)\cap \pa\E(j)\cap \pa \E(0).
	\end{align*}
We note that $B\subseteq A$ is immediate and also $A\subseteq \S_0$ is immediate, since if $x\in A$ then $x\notin \E(0)^{\mez}\supseteq \pared \E(0)$. We are left to show that $\S_0\subseteq B$. In order to do this we define the following family of subsets of $\Om$.
\begin{align}
E_i&:=\displaystyle \mathring{\E(i)} \ \ \ &\text{for all $0 \leq i\leq N$}\label{Partizione di omega in insiemi furbi}\\
F_{i,j}&:=\displaystyle \Om\cap \pa \E(i)\cap \pa\E(j) \setminus \left(\bigcup_{\substack{k=0, \\ k\neq i,j}}^N \pa\E(k)\right), \ \ \ &\text{for all $0\leq i<j \leq N$}\label{Partizione di omega in insiemi furbi1}\\
G_{i,j}&:=\displaystyle  \Om\cap \pa \E(i)\cap \pa\E(j) \cap \pa\E(0), \ \ \ &\text{for all $1\leq i<j \leq N$}.\label{Partizione di omega in insiemi furbi2}
\end{align}
It is easy to verify that the Borel sets defined in \eqref{Partizione di omega in insiemi furbi},\eqref{Partizione di omega in insiemi furbi1},\eqref{Partizione di omega in insiemi furbi2} form a partition of $\Om$. Now, for a given point $x\in \S_0$, clearly $x\notin E_i$ for all $i=0,\ldots,N$. Thus either $x\in F_{i,j}$  for some $0 \leq i < j\leq N$ or $x\in G_{i,j}$ for some $1 \leq i < j\leq N$. If $x\in F_{i,j}$, by closedness  there exists a small ball $B_s(x)$ such that $\pa \E(k)\cap B_s(x)=\emptyset$ for all $k\neq i,j$. This implies that either $i=0$ or $j=0$ (since we have chosen $x\in \S_0\subset \pa\E(0)$) and that $\pa \E(0)\cap B_s(x)=\pa\E(j)\cap B_s(x)$ leading to say that $\pa\E(0)$ must be regular in a small neighborhood of $x$ and contradicting $x\in \S_0$. Hence necessarily $x\in G_{i,j}$ for some  $1 \leq i < j\leq N$ and we achieve the proof: $\S_0\subseteq B$. 
\end{proof}
The following corollary is an easy consequence of Proposition \ref{prima caratterizazione}.
\begin{corollary}\label{egli e chiuso}
Let $n\leq 7$, $N\geq 2$ and $\Om$ be an open, bounded and connected set with $C^{1}$ boundary and finite perimeter in $\R^{n}$. If $\E$ is a Cheeger $N$-cluster for $\Om$, then $\S(\E(0);\Om)$ is closed.
\end{corollary}
\begin{proof}
Thanks to Proposition \ref{prima caratterizazione} we have that
	\begin{equation}\label{putin}
	\S(\E;\Om)=\Om\cap \bigcup_{\substack{i,j=1,\\ i\neq j}}^N \pa \E(i)\cap \pa\E(j)\cap \pa \E(0).
	\end{equation}
Let $\{x_k\}_{k\in \N}\subseteq \S(\E;\Om)$ such that $x_k\rightarrow x$. Up to extract a subsequence we have that $\{x_k\}_{k\in \N}\subset  \Om\cap \pa \E(i)\cap \pa\E(j)\cap \pa \E(0)$ for some $1\leq i<j\leq N$ (since \eqref{putin} is in force). By closedness we obtain $x\in \pa\E(i)\cap \pa\E(j)\cap \pa\E(0)$ and we need to prove that $x\in \Om$. If $x\in \pa\Om$ we have $x\in \pa\E(i)\cap \pa\E(j)\cap \pa\Om$ which is a contradiction since $x$ would be a point of density $\frac{1}{2}$ for three disjoint sets $\E(1),\E(2),\Om^c$). Hence $x\in \Om$ and thus $x\in \S(\E(0);\Om)$.
\end{proof}
The proof of Theorem \ref{mainthm2} is now obtained as an easy consequence of the previous results.
\begin{proof}[Proof of Theorem \ref{mainthm2}]
Follows by Propositions \ref{ognicameraconfinaconilvuoto}, \ref{prima caratterizazione} and by Corollary \ref{egli e chiuso}.
\end{proof}
\section{The planar case} \label{cpt 4 the planar case}
As in the previous sections, the proof of Theorem \ref{mainthm3} is attained by combining different results that we state and prove in Subsection \ref{cpt 4 sbsct proof of mainthm3}. We premise some technical lemmas.
\subsection{Technical lemmas}

\begin{lemma}\label{buone componenti connesse}
Let $n\leq 7$, $N\geq 2$ and $\Om$ be an open, bounded and connected set with $C^{1}$ boundary and finite perimeter in $\R^{n}$. Let $\E$ be a Cheeger $N$-cluster for $\Om$. If $E$ is an indecomposable component of $\E(0)$ such that $E\cc \Om$, then there exist at least three different indexes $i,j,k\neq 0$ such that $\pa E \cap \E(i)\neq \emptyset$, $\pa E\cap \E(j)\neq \emptyset$, and $\pa E\cap \E(k)\neq \emptyset$. In particular, $E$ shares boundary at least with three different chambers.
\end{lemma}
\begin{proof}
Let $E$ be a generic indecomposable component of $\E(0)$. Assume that $E$ shares its boundary with exactly 
two other different chambers $j,k\geq 1$ and $\partial E\cap \pa \Om=\emptyset$. Then either
$$a) \ \ \H^{n-1}(\pa^*E\cap\pa\E(j))\geq \H^{n-1}(\pa^*E\cap\pa\E(k)),$$
or
$$b) \ \ \H^{n-1}(\pa^*E\cap\pa\E(k))\geq \H^{n-1}(\pa^*E\cap\pa\E(j))$$
hold. Assume that $a)$ holds and define $\E_1(j):=\E(j)\cup E$, $\E_1(i)=\E(i)$ for $i\neq j$. Since 
$$P(\E_1(j))=P(\E(j))-\H^{n-1}(\pa^*E\cap\pa\E(j))+\H^{n-1}(\pa^*E\cap\pa\E(k))$$
we obtain:
\begin{eqnarray*}
H_N(\Om)&\leq& \sum_{i=1}^N\frac{P(\E_1(i))}{|\E_1(i)|}\\
&=&\frac{P(\E_1(j))}{|\E_1(j)|}+\sum_{i\neq j}\frac{P(\E(i))}{|\E(i)|}\\
&=&\frac{P(\E(j))-\H^{n-1}(\pa^*E\cap\pa\E(j))+\H^{n-1}(\pa^*E\cap\pa\E(k))}{|\E(j)|+|E|}+\sum_{i\neq j}\frac{P(\E(i))}{|\E(i)|}\\
&\leq&\frac{P(\E(j))}{|\E(j)|+|E|}+\sum_{i\neq j}\frac{P(\E(i))}{|\E(i)|}.
\end{eqnarray*}
If $|E|>0$ we are led to $H_N(\Om)<H_N(\Om)$ which is a contradiction, so $|E|=0$. Since $E$ is open (because $\E(0)$ is open), then $E=\emptyset$. If $E$ shares its boundary with exactly one chamber we argue in the same way. We have discovered that every decomposable component of $\E(0)$ that shares boundary with exactly one or two chambers is empty. The proof is complete.
\end{proof}

\begin{lemma}\label{nel piano!}
Let $E$ be a Cheeger set of an open bounded set $A\subset \R^{2}$. Assume that the following properties hold for $E$:
\begin{itemize}
\item[1)] $\#(\bd_{\pa A}(\pa A\cap\pa E))<+\infty$,
where
	\[
	\bd_{\pa A}(\pa A\cap\pa E)=[\pa A \cap\pa E]\cap\overline{[\pa A\setminus\pa E]};
	\]
\item[2)] every $x\in \pa A\cap \pa E$ is a regular point for $\pa A$, namely $x\in \pa A\cap \pa E \subseteq \pared A$ ;
\end{itemize}
Then $\H^1(\pa E \cap \pa A)>0$.  
\end{lemma}

\begin{remark}
\rm{
It seems that it is possible to generalize Lemma \ref{nel piano!} to dimension $n\geq 2$ by making use of Alexandrov's Theorem \cite{aleksandrov1962uniqueness} for the characterization of the Constant Mean Curvature (CMC) embedded hypersurface in $\R^n$. In this (more technical) framework hypothesis 1) can be weakened. Anyway, since we do not have to deal (at least here) with $n\geq 2$ and since for our purposes Lemma \ref{nel piano!} is all we need to complete the proof of Theorem \ref{mainthm3} we decide to not focus on this generalization.
}
\end{remark}

\begin{proof}[Proof of Lemma \ref{nel piano!}]
Assume by contradiction that $\H^1(\pa E \cap \pa A)=0$. In this case 	
	\[
	\bd_{\pa A}(\pa A\cap\pa E)=\pa E \cap \pa A.
	\] 
Let $F$ be an indecomposable component of $E$ and note that, since $E$ is a Cheeger set for $A$ it must hold
	\begin{equation}\label{e pure lui cheeger}
	\frac{P(F)}{|F|}=h(A).
	\end{equation}
Set 
	\[
	M=\bd_{\pa A}(\pa A\cap\pa F)
	\]
and $\#(M)=k<+\infty$. The well-known regularity theory for Cheeger sets, combined with the fact that $k<+\infty$ tells us that 
 	\[
 	\pa F\cap A=\bigcup_{i=1}^k \a_i
 	\] 
where each $\a_i$ is a piece of the boundary of a suitable ball $B_i$ (relatively open inside $\pa B_i$) of radius $\frac{1}{h(A)}$. The finiteness of $M$ implies that for a suitably small $r$ it holds $B_r(x)\cap M=\{x\}$ for all $x\in M$ and this means that for every $x\in M$ there exists two arcs $\a_i,\a_j$ (with possibly $i=j$) such that $x\in \overline{\a_i}\cap \overline{\a_j}$. \\
Let $B_i, B_j$ the balls from which such arcs come from: $\a_i\in \pa B_i$, $\a_j \in \pa B_j$. Hypothesis $2)$ implies that the outer unit normal to $B_i$ and to $B_j$ at $x$ must coincide with $\nu_A(x)$ and hence the balls $B_i$ and $B_j$ must coincide as well. Since $k<+\infty$, by iterating this argument we conclude that there exists only one ball $B$ of radius $\frac{1}{h(A)}$ such that $M\subset \pa B$ and $\pa F\cap A= \pa B\cap A$. In particular $\pa F$ is equal to $\pa B$ and by exploiting \eqref{e pure lui cheeger} we bump into a contradiction
\begin{eqnarray*}
\frac{P(F)}{|F|}&=&h(A)\\
\frac{P(B)}{|B|}&=&h(A)\\
\frac{\frac{2\pi}{h(A)}}{\frac{\pi}{h(A)^2}}&=&h(A)\\
\frac{\frac{2}{h(A)}}{\frac{1}{h(A)^2}}&=&h(A)\\
2h(A)&=&h(A).
\end{eqnarray*}
The contradiction comes from the fact that we have assumed $\H^1(\pa E \cap \pa A)=0$, hence $\H^1(\pa E \cap \pa A)>0$ and the proof is complete.\\
\end{proof}

\subsection{Proof of Theorem \ref{mainthm3}} \label{cpt 4 sbsct proof of mainthm3}
\begin{proposition}\label{struttura singolaritapiano}
Let $N\geq 2$ and $\Om$ be an open, bounded and connected set with $C^{1}$ boundary and finite perimeter in the plane.  Let $\E$ be a Cheeger $N$-cluster of $\Om$. Then $\S(\E;\Om)=\S(\E(0);\Om)$ is a finite union of points.
\end{proposition}

\begin{proof}
We prove that $\S(\E(0);\Om)$ has no accumulation point. In this way we show that $\S(\E(0);\Om)$ is a closed (thanks to Theorem \ref{mainthm2}), bounded set of $\R^2$ (since $\Om$ is bounded) without accumulation points which means that $\S(\E(0);\Om)$ must be a finite union of points.\\

Set $\S_0=\S(\E(0);\Om)$ for the sake of brevity. Let $\xi\in \pa\E(i)\cap \pa\E(j)$ for some $1\leq i<j\leq N$ that without loss of generality we assume to be $i=1,j=2$. We can assume (up to a translation) also that $\xi=(0,0)$. Since $\pa\E(1),\pa\E(2)$ are regular, up to a rotation we can find a small closed cube 
	\[
	Q_{\e}:=[-\e,\e]\times[-\e,\e]\cc\Omega
	\]
centered at $\xi=(0,0)$ and two $C^1$ functions $f_1,f_2:[-\e,\e]\rightarrow \R$ such that $f_1(x)\leq f_2(x)$ for all $x\in[-\e,\e]$ and:
\begin{align*}
\E(1)\cap Q_{\e}&=\{(x,y) \in Q_{\e} \ | \ -\e \leq y\leq  f_1(x) \},\\
\pa \E(1)\cap Q_{\e} &=\{(x,f_1(x)) \ | \ x\in [-\e,\e]\},\\
\E(2)\cap Q_{\e}&=\{(x,y) \in Q_{\e} \ | \ \  f_2(x)\leq y \leq \e  \},\\
\pa \E(2)\cap Q_{\e}&=\{(x,f_2(x)) \ | \ x\in [-\e,\e] \}\\
\pa \E(2)\cap \pa \E(1)\cap Q_{\e}&=\{(x,y) \in Q_{\e} \ | \ \ y=f_1(x)=f_2(x)\leq \e  \},\\
\E(0)\cap Q_{\e}&=\{(x,y)\in Q_{\e} \ | \ -\e\leq f_1(x)< y< f_2(x)\leq \e\},\\
\E(k)\cap Q_{\e}&=\emptyset \ \ \text{for all $k\geq 3$},
\end{align*}
(see Figure \ref{compoconne}). Since the blow-up of $\pa\E(1)\cap\pa\E(2)$ at $\xi=(0,0)$ is a line, up to further decrease $\e$, we can also assume that $\E(1)\cap Q_{\e}$ and $\E(2)\cap Q_{\e}$ are indecomposable, which is equivalent to say:
	\[
	|f_1(x)|<\e, \ \ \ |f_2(x)|<\e \ \ \ \ \forall \ x\in [-\e,\e].
	\]
We consider the set
	\begin{align*}
	E_0&:= \{ x\in [-\e,\e] \ | \ f_1(x)<f_2(x) \}.
	\end{align*}
which is relatively open inside $[-\e,\e]$ (is the counter-image of the open set $(-2\e,0)$ through the continuous function $f_1-f_2$).
\begin{figure}
\begin{center}
 \includegraphics[scale=1]{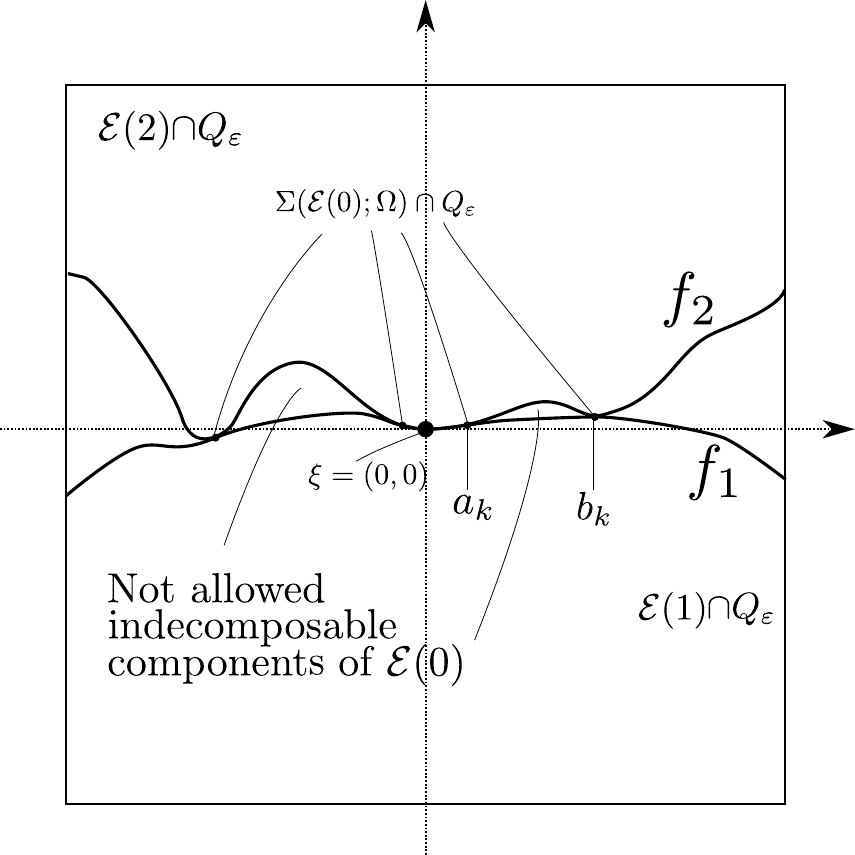}\caption{{\small This kind of behavior contradicts the minimality property of $\E$, in particular it contradicts Lemma \ref{buone componenti connesse}.}}\label{compoconne}
 \end{center}
\end{figure}
 Hence, $E_0$ must be the union of countably many disjoint (open) intervals:
	\[
	E_0=[-\e,a)\cup (b,\e] \cup \left(\bigcup_{k=2}^{+\infty}(a_k,b_k)\right)
	\]
for $\{a_k\}_{k=1}^+{\infty},\{b_k\}_{k=1}^+{\infty}\subset [-\e,\e]$  (with a slight abuse of notation we are allowing also the possible cases $a=-\e$, $b=\e$ or even $a=a_1$,$b_{\infty}=b$ as in Figure \ref{compoconne}). It is immediate that each
	\[
	A_k:=\{(x,y)\in Q_{\e} \ | \ a_k<x<b_k, \ f_1(x) < y <f_2(x) \}
	\]
is an indecomposable component of $\E(0)$. Observe that $E_k\cc Q_{\e} \cc \Om$ is an indecomposable component of $\E(0)$ that share its boundary with exactly two chambers ($\E(1), \E(2)$) and hence contradicts Lemma \ref{buone componenti connesse}. This means that the only possibility is
	\[
	E_0=[-\e,a)\cup (b,\e] 
	\]
for some $a,b\in [-\e,\e]$. By possibly decreasing $\e$ we can assume that $(-\e,f_1(-\e) ),(\e,f_1(\e) )\notin \S_0\cap Q_{\e}$. The only possibilities remained are
\begin{itemize}
\item[1)] $a=-\e$ and $b=\e$, thus $\S_0\cap Q_{\e}=\emptyset$;
\item[2)] $a\neq -\e$ and $b = \e$, thus $\S_0\cap Q_{\e}=\{(a,f_1(a))\}=\{(a,f_2(a))\}$;
\item[3)] $a=-\e$ and $b\neq \e$, thus $\S_0\cap Q_{\e}=\{(b,f_1(b))\}=\{(b,f_2(b))\}$;
\item[4)] $a\neq -\e$ and $b\neq \e$, thus $\S_0\cap Q_{\e}=\{(a,f_1(a)), (b,f_1(b))\}=\{(a,f_2(a)), (b,f_2(b))\}$.
\end{itemize} 
In particular $\#(\S_0\cap Q_{\e})\leq 2$ which means that $\S_0$ has no accumulation points.
\end{proof}

We now exploit the stationarity of Cheeger $N$-clusters in order to derive information on their structure. 

\begin{proposition} \label{curvature}
Let $N\geq 2$ and $\Om$ be an open, bounded and connected set with $C^{1}$ boundary and finite perimeter in the plane. Let $\E$ be a Cheeger $N$-cluster of $\Om$. For every $j,k=0,\ldots,N$, $k \neq j$ the set
$$E_{j,k}:=[\pa \E(j)\cap \pa \E(k) \cap\Om]\setminus \S(\E(0);\Om)$$ 
is relatively open in $\pa \E(j)\cap\pa\E(k)\cap \Om$ and is the finite union of segments and circular arcs. Moreover the set $\E(j)$ has constant curvature $C_{j,k}$ on each open set $A$ such that $A\cap \pa \E(j) \subseteq E_{j,k}$. The constant $C_{j,k}$ is equal to:
\begin{equation}
C_{j,k}= \left\{
\begin{array}{cc}
\frac{|\E(k)|h(\E(j))-|\E(j)|h(\E(k))}{|\E(j)|+|\E(k)|}, & \text{if $k\ne 0$}\\
 & \\
h(\E(j)), & \text{if $k=0$}.
\end{array}
\right.
\end{equation}
As a consequence the set $\E(k)$ has constant curvature $C_{k,j}=-C_{j,k}$ on each open set $A$ such that $A\cap \pa \E(k)\subseteq  E_{k,j}(=E_{j,k})$.
\end{proposition}

\begin{proof}
If $k=0$ (or $j=0$) we just notice that $\E(j)$ is a Cheeger set for 
	\[
	\Om_0=\Om\setminus \bigcup_{\substack{i=1,\\ i\neq j}}^{N} \E(i)
	\]
so the free boundary $E_{j,0}$ is the finite union of segments and circular arcs and $\E(j)$ has constant mean curvature $C_{j,0}=h(\E(j))$ on each open set $A$ such that $A\cap \pared \E(j)\subseteq E_{j,0}$.\\

Thus, we consider a couple $j,k\in \{1,\ldots, N\}$ such that 
	\[
	[\pa\E(i)\cap\pa\E(k)\cap \Om]\setminus  \S(\E(0);\Om)\neq \emptyset
	\]
(otherwise there is nothing to prove and the proposition is trivial). The set $\S(\E(0);\Om)$ is closed and is the finite union of points (thanks to Lemma \ref{struttura singolaritapiano}). Hence 
	\[
	E_{j,k}:=[\pa \E(j)\cap \pa \E(k) \cap\Om]\setminus \S(\E(0);\Om)
	\]
is relatively open in $\pa \E(j)\cap\pa\E(k) \cap \Om$. For every $x\in E_{j,k}$, by closedness, there exists a ball $B_r(x)$ such that 
	\[
	B_r(x)\cap  \E(i)=\emptyset \ \ \ \ \forall \ i=1,\ldots,N, \  \ i\neq j,k.
	\]
Note that, up to further decrease the value of $r$ it must hold as well
	\[
	B_r(x)\cap \E(0)=\emptyset.
	\]
Indeed if this is not the case, we would have that $x\in \pa \E(0)\cap \E(0)^{\zero}$ and thus (thanks to Proposition \ref{prima caratterizazione}) $x\in \S(\E(0);\Om)$ which is a contradiction since $x\in E_{j,k}$. Hence, because of the minimality of $\E$, the set $\pa\E(j)\cap \pa\E(k)\cap B_r(x)$ must solve an isoperimetric problem with volume constraint inside $B_r(x)$ and by exploiting stationarity it is possible to prove that each solution to an isoperimetric problem with volume constraint must be an analytic constant mean curvature hypersurface (\cite[Theorems 17.20, 24.4 ]{maggibook}). Set $C_{j,k}$ and $C_{k,j}$ to be respectively the value of the mean curvature of $\E(j)$ and of $\E(k)$ in $B_r(x)$. Observe that, since $\E(k)\cap B_r= \E(j)^c\cap B_r$ it holds trivially that $C_{k,j}=-C_{j,k}$. Let us compute the (constant) value of $C_{j,k}$.\\

Consider a map $T\in C^{\infty}_c(B_r;\R^2)$, define for all $|t|<\e$ the diffeomorphism $f_t(y)=y+t T(y)$ and the cluster $\E_t:=\{f_t(\E(i))\}_{i=1}^N$. Of course, for $t$ suitably small, $\E_t\Delta \E \subset\subset B_r $. Note that $\{f_t \ | \ -\e<t<\e \}$ is a local variation in $B_r$ and that $T$ is its initial velocity. By minimality it holds: 
$$\frac{P(\E(j))}{|\E(j)|}+\frac{P(\E(k))}{|\E(k)|} \leq \frac{P(\E_t(j))}{|\E_t (j)|}+\frac{P(\E_t(k))}{|\E_t (k)|}, \ \ \ \forall \ |t|<\e.$$
Thus  
\begin{equation}\label{mi sparooo}
0\leq \frac{d}{dt}\Big{|}_{t=0}\frac{P(\E_t(j))}{|\E_t (j)|}+\frac{d}{dt}\Big{|}_{t=0} \frac{P(\E_t(k))}{|\E_t (k)|}.
\end{equation}
With some easy computations
\begin{eqnarray*}
\frac{d}{dt}\Big{|}_{t=0}\frac{P(\E_t(j))}{|\E_t (j)|} &=&\frac{|\E(j)|\frac{d}{dt}\Big{|}_{t=0}P(\E_t(j))-P(\E(j))\frac{d}{dt}\Big{|}_{t=0}|\E_t(j)|}{|\E (j)|^2},\\
\text{}\\
\frac{d}{dt}\Big{|}_{t=0}P(\E_t(j)) &=&C_{j,k} \int_{\pa \E(j)\cap \pa \E(k) \cap B_r}(T(y)\cdot \nu_{\E(j)}(y))  \d \H^1(y)\\
\frac{d}{dt}\Big{|}_{t=0}|\E_t(j)| &=& \int_{\pa \E(j)\cap \pa \E(k) \cap B_r}(T(y)\cdot \nu_{\E(j)}(y)) \d \H^1(y).
\end{eqnarray*}
where we have used the facts that the mean curvature exists and that it is constantly equal to $C_{j,k}$ in $B_r$ (and hence on $\pa\E(j)\cap \pa \E(k)\cap B_r$). By denoting with 
$$ f_j=\int_{\pa \E(j)\cap \pa \E(k) \cap B_r}(T(y)\cdot \nu_{\E(j)}(y))  \d \H^1(y)$$
$$ f_k=\int_{\pa \E(j)\cap \pa \E(k) \cap B_r}(T(y)\cdot \nu_{\E(k)}(y))  \d \H^1(y),$$
we can write:
\begin{eqnarray*}
\frac{d}{dt}\Big{|}_{t=0}\frac{P(\E_t(j))}{|\E_t (j)|} &=&\frac{|\E(j)| f_j C_{j,k}-P(\E(j))f_j}{|\E (j)|^2},\\
\frac{d}{dt}\Big{|}_{t=0}\frac{P(\E_t(k))}{|\E_t (k)|} &=&\frac{|\E(k)| f_k C_{k,j}-P(\E(k))f_k}{|\E (k)|^2},\\
\end{eqnarray*}
that plugged into \eqref{mi sparooo},by observing that $f_j=-f_k$, lead to the relation:
\begin{eqnarray*}
0&\leq &\frac{d}{dt}\Big{|}_{t=0}\frac{P(\E_t(j))}{|\E_t (j)|} +\frac{d}{dt}\Big{|}_{t=0}\frac{P(\E_t(k))}{|\E_t (k)|} \\
&=&\frac{|\E(i)| f_j C_{j,k}-P(\E(j))f_j}{|\E(j)|^2}+\frac{|\E(k)| f_k C_{k,j}-P(\E(k))f_k}{|\E(k)|^2}\\
&=&f_j\left[ \frac{|\E(j)|  C_{j,k}-P(\E(j))}{|\E (j)|^2}-\frac{|\E(k)|  C_{k,j}-P(\E(k))}{|\E (k)|^2}\right].
\end{eqnarray*}
By choosing a $T_1$ such that $f_j$ is positive and then a $T_2$ such that $f_j$ is negative we conclude that 
\begin{eqnarray*}
0&=&\frac{|\E(j)|  C_{j,k}-P(\E(j))}{|\E (j)|^2}-\frac{|\E(k)|  C_{k,j}-P(\E(k))}{|\E (k)|^2}.\\
\end{eqnarray*}
Finally, by exploiting $C_{j,k}=-C_{k,j}$ we rach
\begin{eqnarray*}
0&=&\frac{ C_{j,k}}{|\E (j)|}-\frac{P(\E(j))}{|\E (j)|^2}-\frac{  C_{k,j}}{|\E (k)|}+\frac{P(\E(k))}{|\E (k)|^2}\\
&=&\frac{ C_{j,k}}{|\E (j)|}-\frac{P(\E(j))}{|\E (j)|^2}+\frac{  C_{j,k}}{|\E (k)|}+\frac{P(\E(k))}{|\E (k)|^2},
\end{eqnarray*}
that can be re-arranged as:
\begin{eqnarray*}
C_{i,k}(|\E(j)|+|\E(k)|)&=&h(\E(j))|\E (k)|-h(\E(k))|\E (j)|,\\
C_{j,k}&=&\frac{h(\E(j))|\E (k)|-h(\E(k))|\E (j)|}{|\E(j)|+|\E(k)|}.
\end{eqnarray*}
In particular, since $C_{j,k}$ do not depend on $x\in E_{j,k}$ and since the ambient space dimension is $n=2$, $E_{j,k}$ must be a finite union of circular arcs or segments with curvature $|C_{j,k}|$.
\end{proof}

Our last proposition of the section put together Lemma \ref{nel piano!} Proposition \ref{struttura singolaritapiano} and Proposition \ref{curvature} and tells us that the interior chambers of a Cheeger $N$-cluster are always indecomposable. We are making strong use of Proposition \ref{struttura singolaritapiano} which does not holds on $\pa\Om$ (see Figure \ref{controfinitezza} and Remark \ref{remark controfinitezza}) and that is why we cannot extend the proof of the Proposition \ref{connessionecamere} to all the chambers. 

\begin{proposition}\label{connessionecamere}
Let $N\geq 2$ and $\Om$ be an open, bounded and connected set with $C^{1}$ boundary and finite perimeter in the plane. Let $\E$ be a Cheeger $N$-cluster for $\Om$. Then, every chamber $\E(i)\cc \Om$ for $i\neq 0$ is indecomposable. 
\end{proposition}

\begin{proof}
Assume, without loss of generality $i=1$ and let $E_1$ and $E_2$ be two different components of $\E(1)$. By minimality it must hold
	\begin{equation}\label{mi impiccoooooo}
	\frac{P(E_1)}{|E_1|}=\frac{P(E_2)}{|E_2|}=\frac{P(\E(1))}{|\E(1)|}.
	\end{equation}
The component $E_2$ is a Cheeger set for 
	\[
	A=\left(\bigcup_{j\neq 1} \Om\setminus \E(j)\right) \cup E_1
	\] 
and by Theorem \ref{mainthm1}, every $x\in \pa E_2\cap \pa A$ is a regular point for $\pa A$. Moreover $\bd_{\pa A}(\pa A \cap\pa E_2)\subseteq \S(\E(0);\Om)$ (since $\E(i)\cc \Om$) and thus, thanks to Proposition \ref{struttura singolaritapiano}, we have 
	\[
	\#(\bd_{\pa A}(\pa A \cap\pa E_2))\leq \#(\S(\E(0);\Om))<+\infty.
	\]
Therefore we can exploit Lemma \ref{nel piano!} on $E_2$ and conclude that $\H^1(\pa E_2 \cap \pa A)>0$. In particular we deduce that there exists an index $k\neq 0,1$ such that $\H^1(\pa E_2\cap \pa \E(k))>0$. Define the new cluster $\F(1)=E_1$, $\F(j)=\E(j)$ for $j\neq 1$ (see Figure \ref{lemmaindecompo}). Thanks to \eqref{mi impiccoooooo} it holds:
\begin{equation}
H_N(\Om)=\sum_{i=1}^N \frac{P(\F(i))}{|\F(i)|}.
\end{equation}
Hence $\F$ it is also a Cheeger $N$-cluster for $\Om$. Consider the piece of boundary 
	\[
	S=[\pa \E(k)\cap \pa E_2]\setminus \S(\E(0);\Om) \neq \emptyset
	\]
from the old cluster $\E$.  Proposition \ref{curvature} tells us that $S$ must be a circular arc and that $\E(k)$ must has constant mean curvature on $S$ equal to:
 $$C_{k,1}=\frac{|\E(1)|h(\E(k))-|\E(k)|h(\E(1))}{|\E(1)|+|\E(k)|}.$$
 \begin{figure}
\begin{center}
 \includegraphics[scale=2.2]{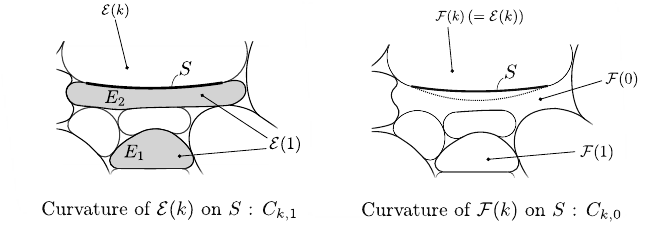}\caption{{\small }}\label{lemmaindecompo}
 \end{center}
\end{figure}
From the other side it holds $\F(k)=\E(k)$ and, since $S$ is now a part of the free boundary of $\F(k)$ (we have removed the component $E_2$), we have that $\F(k)=\E(k)$ must has constant mean curvature on $S$ also equal to:
$$C_{k,0}=h(\F(k))=h(\E(k)).$$
Thus equality $C_{k,1}=C_{k,0}$ must be in force, implying $(h(\E(k))+h(\E(1)))|\E(k)|=0$ which is impossible. 
\end{proof}

\begin{proof}[Proof of Theorem \ref{mainthm3}]
It follows from Propositions \ref{struttura singolaritapiano}, \ref{curvature} and \ref{connessionecamere}.
\end{proof}
\section{The limit of $\Lambda_N^{(p)}$ as $p$ goes to one }\label{limite}
We conclude this paper by focusing on the asymptotic trend of $H_N$. We first briefly state the following Theorem involving the existence of the optimal partition for problem \eqref{p-laplacian}.
\begin{theorem}
For every $1<p\leq n$ there exists an optimal partition for $\Om$ in quasi-open sets $\{\Om_i\}_{i=1}^N$ such that 
$$\Lambda_N^{(p)}(\Om)=\sum_{i=1}^N \lambda_1^{(p)}(\Om_i).$$
\end{theorem}
\begin{proof}
The existence of an optimal partition for $\Lambda_N^{(p)}(\Om) $ follow as a simple variation of the argument in \cite{CaLi07}, or 
as a consequence of more general results contained in \cite{BucBuH98}, \cite{BucVel13} or \cite{BuDM93} and thus we omit the details. 
\end{proof}
In the following Proposition we compute the limit  of $\Lambda_N^{(p)}$ as $p$ goes to one.
\begin{proposition}\label{limit}
If $\Om$ is an open bounded set with $C^1$ boundary then
$$\lim_{p\rightarrow 1} \Lambda_N^{(p)}(\Om)=H_N(\Om).$$
\end{proposition}
\begin{proof}
Let $\E$ be a Cheeger $N$-cluster for $\Om$. Since $\pa \E(i)$ is $C^1$, for every $i=1,\ldots,N$ there exists a sequence of open sets $\{\E_t(i)\}_{t>0}$ such that $\E_t(i)\cc \E(i)$ for all $t>0$ and
$$\E_{t}(i)\rightarrow \E(i) \text{  in $L^{1}$}, \ \ \ P(\E_{t}(i))\rightarrow P(\E(i)),$$ 
as $t\to 0$ (see \cite{Schmidt2015}). In this way, since $\E_t(i)$ are open sets (and thus quasi-open sets) strictly contained into $\Om$ and with disjoint closure, by exploiting \eqref{eqn limite per p che tende a uno} we reach: 
\begin{align*}
H_N(\Om)&= \lim_{t\rightarrow 0} \sum_{i=1}^{N}\frac{P(\E_{t}(i))}{|\E_{t}(i)|}\geq  \lim_{t\rightarrow 0} \sum_{i=1}^{N}h(\E_{t}(i))\\
&\geq  \lim_{t\rightarrow 0} \limsup_{p\rightarrow 1}\sum_{i=1}^{N}\lambda_N^{(p)}(\E_{t}(i))\geq \limsup_{p\rightarrow 1}\Lambda_N^{(p)}(\Om).
\end{align*}
On the other hand, thanks to \eqref{lb p-eig} and to Jensen's inequality we get \eqref{Lp lowerbound}:
\begin{align*}
\sum_{j=1}^N \lambda_1^{(p)}(\E(i))&\geq  \sum_{j=1}^N \left(\frac{h(\E(i))}{p}\right)^{p} \geq \frac{1}{N^{p-1}}\left(\sum_{j=1}^N \frac{h(\E(i))}{p}\right)^{p}\\
&\geq\frac{1}{N^{p-1}}\left( \frac{H_N(\Om)}{p}\right)^{p}
\end{align*}

which completes the proof.
\end{proof}
\subsection{On the asymptotic behavior of $H_N$ in dimension $n=2$}\label{asymptoyc of HN}

\begin{theorem}\label{asymptotic behavior}
Denote with $B$ a ball of unit radius and with $H$ a unit-area regular hexagon. Let $N\geq 2$ and $\Om$ be an open, bounded and connected set with $C^{1}$ boundary and finite perimeter in the plane. Then the following assertions hold true:
\begin{itemize}
\item[1)] If $\E$ is a Cheeger $(N+1)$-cluster for $\Om$ then:
	\[
	|\E(i)|\geq \frac{h(B)^2\pi}{(H_{N+1}(\Om)-H_N(\Om))^2}  \ \ \ \forall \ i=1,\ldots,N+1;
	\]
\item[2)] $\displaystyle H_N(\Omega)+\frac{h(B)\sqrt{\pi}}{\sqrt{|\Omega|}}\sqrt{N+1}\leq H_{N+1}(\Omega),$  for all $N\in \N$;
\item[3)] for every $0\leq \e<\frac{1}{2}$ there exists $N_0(\Om,\e)$ such that:
	\[
	\frac{\sqrt{\pi} h(B)}{\sqrt{|\Omega|}}N^{\frac{3}{2}}\leq H_N(\Omega)\leq \frac{h(H)}{\sqrt{|\Omega|}}N^{\frac{3}{2}}+ N^{\frac{3}{2}-\e} \ \ \ \  \text{for all $N\geq N_0(\Om,\e)$}.
	\]
\end{itemize}
\end{theorem}
\begin{proof}
Thanks to the planar Cheeger inequality 
\begin{equation}\label{CheegerInequality}
h(E)\geq \sqrt{\pi}\frac{h(B)}{\sqrt{|E|}}
\end{equation}
we observe that, given $\E$ a Cheeger $(N+1)$-cluster of $\Om$, it holds:
\begin{align*}
H_{N+1}(\Om)&= \sum_{i=1}^{N+1} \frac{P(\E(i))}{|\E(i)|} \geq\frac{P(\E(j))}{|\E(j)|} +\sum_{i=1, i\neq j}^{N+1} \frac{P(\E(i))}{|\E(i)|}\\
&\geq h(\E(j))+H_{N}(\Om)\geq \frac{\sqrt{\pi}h(B)}{\sqrt{|\E(j)|}}+H_{N}(\Om)
\end{align*}
which, implies Property 1).\\

Property 2) follows from Property 1):
\begin{eqnarray*}
|\Om|-\frac{(N+1)h(B)^2\pi }{(H_{N+1}(\Om)-H_N({\Om}))^2} &\geq&|\Om|-\sum_{i=1}^{N+1}|\E(i)|\geq 0
\end{eqnarray*}
and so
\begin{eqnarray*}
|\Om|&\geq& \frac{(N+1)h(B)^2\pi}{(H_{N+1}(\Om)-H_N({\Om}))^2},
\end{eqnarray*}
which implies
\begin{eqnarray*}
(H_{N+1}(\Om)-H_N({\Om}))^2&\geq& \frac{(N+1)h(B)^2\pi}{|\Om|}
\end{eqnarray*}
and thus
\begin{eqnarray*}
H_{N+1}(\Om)&\geq& H_N({\Om})+ \sqrt{(N+1)}\frac{\sqrt{\pi}h(B)}{\sqrt{|\Om|}}.
\end{eqnarray*}
Let us prove Property 3). Let $\E$ be a Cheeger $N$-cluster for $\Om$. We exploit again  the Cheeger inequality \eqref{CheegerInequality} and
we obtain the lower bound
\begin{align*}
H_N(\Om)&=\sum_{i=1}^{N}h(\E(i))\geq  \sqrt{\pi} h(B)\sum_{i=1}^{N}\frac{1}{\sqrt{|\E(i)|}}\\
&\geq   \sqrt{\pi} h(B) N^{\frac{3}{2}}\left(\frac{1}{\sum_{i=1}^{N}|\E(i)|}\right)^{\frac{1}{2}}\geq   \sqrt{\pi} \frac{h(B)}{\sqrt{|\Om |}} N^{\frac{3}{2}}. 
\end{align*} 
Here we have used the inequality
$$\sum_{i=1}\frac{1}{x_i^{\frac{1}{n}}}\geq N^{\frac{n+1}{n}}\left(\frac{1}{\sum_{i=1}^Nx_i}\right)^{\frac{1}{n}}, \ \ \forall \ N,n\geq 2, \  x_i >0.$$
Let us focus on the upper bound. Let $\H_{\de}$ be the standard hexagonal grid of the plane, made by hexagons of area $\de$ (the one depicted in Figure \ref{fig reference}), labeled with natural numbers. 
\begin{figure}
\begin{center}
 \includegraphics[scale=0.7]{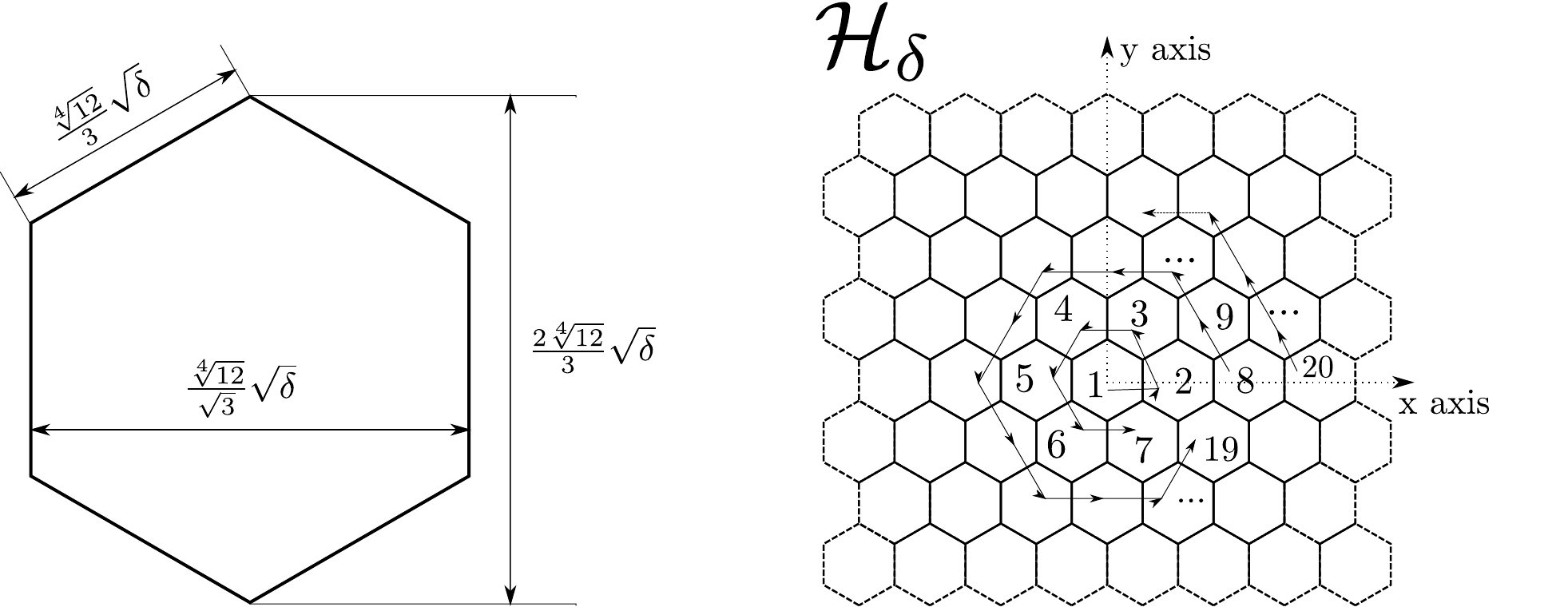}\caption{{\small The standard hexagonal grid of the plane, made by hexagons of area $\de$, together with one of its possible labeling.}}\label{fig reference}
 \end{center}
\end{figure}
Define
	\begin{align*}
	I(\de)&:=\{i\in \N \ | \ \H_{\de}(i)\cc \Om\},\\
	k(\de)&:=\#(I(\de) ).
	\end{align*}
Up to a relabeling, let us assume that $I(\de)=\{1,\ldots,k(\de)\}$. Note that since $\H_{\de}(i)\cc  \Om$ we get
	\[
	H_{k(\de)}(\Om)\leq\sum_{i=1}^{k(\de)} h(\sqrt{\delta} H)=\frac{k(\de)}{\sqrt{\de}}h(H).
	\]
From $\H_{\de}(i) \subset \Om$ for all $i=1,\ldots,k(\de)$ it follows
$$k(\de) \leq \frac{|\Om|}{\de }.$$
If we set $\de(N)=\frac{|\Om|}{N}-\frac{|\Om|}{N^{\a}}$ for some $\a>1$ to be chosen, we are led to 
	\begin{equation}\label{forse stasera finisco}
	H_{k(N)}(\Om)\leq \frac{N^{\frac{3}{2}}}{\sqrt{|\Om|}\left(1-N^{1-\a}\right)^{\frac{3}{2}}} h(H).
	\end{equation}
where $k(N)=k(\de(N))$.
Note that, by setting
	\[
	(\pa \Omega)_{r(N)}:=\pa\Om+B_{r(N)}
	\]
where $r(N)=\sqrt{\de(N)}\diam(H)$, it must hold 
$$\left(\Om\setminus \bigcup_{i=1}^{k(N)} \H(i)\right)\subseteq (\pa \Omega)_{r(N)}.$$
Since $\Om$ has Lipschitz boundary, for $N$ bigger than $N_0(\Om)$, it also holds that
	\[
	| (\pa \Omega)_{r(N)}|\leq 4 r(N) P(\Om)
	\]
and so:
\begin{align*}
|\Om|-\de(N) k(N) &\leq  |(\pa \Omega)_{r(N)}| \leq  4r(N)P(\Om) = 4\sqrt{\de(N)}\diam(H)P(\Om),\\
\end{align*}
that imply
\begin{align*}
 k(N) &\geq  \frac{N}{1-N^{1-\a}} - 4\sqrt{N}\frac{P(\Om)\diam(H)}{\sqrt{|\Om|} \sqrt{1-N^{1-\a}}} .
\end{align*}
For all $N$ bigger than some fixed $N_0$ depending only on $\Om$. It is easy to show that, for all $\a<\frac{3}{2}$, up to further increase $N_0$ in dependence only on $\Om$ and $\a$, it holds 
	\[
	\frac{N}{1-N^{1-\a}} - 4\sqrt{N}\diam(H)\frac{P(\Om)}{\sqrt{|\Om|} \sqrt{1-N^{1-\a}}} \geq N.
	\]
Hence by choosing $\a <\frac{3}{2}$ we obtain
	\[
	k(N)\geq N \ \ \ \ \ \forall \ N\geq N_0,
	\]
and, thanks to the monotonicity given by Property 2) and to \eqref{forse stasera finisco}, provided also $\a>1+\e$ we reach:
	\begin{align*}
	H_N(\Om) &\leq H_{k(N)}(\Om)\\
	&\leq \frac{N^{\frac{3}{2}}}{\sqrt{|\Om|}\left(1-N^{1-\a}\right)^{\frac{3}{2}}} h(H) \\
	&\leq \frac{h(H)}{\sqrt{|\Om|}}  (N^{\frac{3}{2}}+N^{\frac{3}{2}-\e})  \ \ \ \ \ \text{for all $N>N_0(\Om,\e)$}.
	\end{align*}
\end{proof}
 
\bibliography{referencescopia}
\bibliographystyle{alpha}


\end{document}